\documentclass[10pt]{elsarticle}
\journal{ }
\makeatletter
\def\ps@pprintTitle{%
  \let\@oddhead\@empty
  \let\@evenhead\@empty
  \let\@oddfoot\@empty
  \let\@evenfoot\@oddfoot
}
\makeatother
\usepackage[numbers]{natbib}
\usepackage{adjustbox}
\usepackage{graphicx,url}   
\usepackage{xurl}
\usepackage{etoolbox}

\usepackage{setspace} 
\usepackage{geometry} 
\usepackage{eurosym}         
\usepackage{makecell}
   \usepackage{hyperref}
\usepackage{orcidlink}
\usepackage[section]{placeins}
\usepackage{epsfig} 
\usepackage{enumitem} \usepackage{amsmath,amssymb,dsfont}
\usepackage{xcolor}
\usepackage{algorithm}
\usepackage{algpseudocode}
\usepackage{amsthm}
\usepackage{comment}
\usepackage{booktabs}
\usepackage{graphicx}     
\usepackage{caption}      
\usepackage{tabularx} 
\usepackage{subcaption}   
\usepackage{stfloats}
\definecolor{green1}{rgb}{0.0, 0.5, 0.0}
\newcommand{\one}{{\mathbf{1}}}

\newcommand{\tran}{^{\top}}

\newcommand{\E}{{\mathrm E}}

\newcommand{\Pmax}{{P_{\mathrm{max}}^{\mathrm{nt}}}}
\newcommand{\wt}{\widetilde}
\newcommand{\R}{\mathbb{R}}
\newcommand{\N}{\mathbb{N}}

\newcommand{\beq}{\begin{equation}}
\newcommand{\eeq}{\end{equation}}
\newcommand{\bea}{\begin{eqnarray}}
\newcommand{\eea}{\end{eqnarray}}
\newcommand{\beas}{\begin{eqnarray*}}
\newcommand{\eeas}{\end{eqnarray*}}
\newcommand{\ba}{\begin{array}}
\newcommand{\ea}{\end{array}}
\newcommand{\bit}{\begin{itemize}}
\newcommand{\eit}{\end{itemize}}
\newcommand{\ben}{\begin{enumerate}}
\newcommand{\een}{\end{enumerate}}

\newcommand{\ped}[1]{ _{ {\mathrm{#1} } }}

\newcommand{\Real}[1]{ { {\mathbb R}^{#1} } }

\newtheorem{proposition}{Proposition}[section]

\newtheorem{lemma}{Lemma}[section]

\newcommand{\cI}{{\cal I}}

\newcommand{\calT}{{\cal T}}

\newcommand{\calK}{{\cal K}}

\newcommand{\cblue}[1]{\textcolor{blue}{#1}}

\allowdisplaybreaks
\begin{document}
\begin{frontmatter}
\title{Class-Based Smart Charging Control
for Electric Vehicles}
\author[1]{Giuseppe C.~Calafiore}
\author[1]{Luca Ambrosino}
\author[1]{Matteo {Della Rossa}}
\author[2]{Laurent {El Ghaoui}}

\affiliation[1]{Department of Electronics and Telecommunications,
Politecnico di Torino,
Corso Duca degli Abruzzi 24,
10129 Torino, Italy}
\affiliation[2]{College of Engineering and Computer Science and Center of Environmental Intelligence,
VinUniversity,
Hanoi, Vietnam}

\begin{abstract}
This paper proposes a stochastic control framework for the operation of electric-vehicle (EV) charging stations equipped with on-site photovoltaic (PV) generation and battery storage. To maintain scalability for large fleets, vehicles are aggregated into a finite number of classes according to their residual charging demand, which yields a compact state description and avoids vehicle-level optimization. The resulting charging-station dynamics are modeled in discrete time and capture stochastic arrivals, charging-induced class transitions, stochastic departures, PV generation, battery operation, and power exchange with the grid. On the basis of the corresponding expectation model, we formulate a finite-horizon smart-charging problem that jointly optimizes class-wise charging actions and energy-management variables so as to balance minimization of electricity-purchase cost and promotion of transitions to lower residual-demand classes, which is related to customer satisfaction. The stage problem is a linear program, which is solved in shrinking-horizon form and implemented online after integer discretization of the first charging action. We also discuss a robust counterpart that preserves feasibility under interval uncertainty on the first moments of arrivals, departures and PV generation. To validate the framework's scalability and robustness in its intended large-fleet regime, we conduct an extensive simulation campaign across nine distinct configurations, combining three real-world electricity price patterns with three diverse EV-arrival profiles. Across 900 stochastic scenarios, numerical results demonstrate the core advantage of the proposed approach: at comparable but slightly decreased service levels, the controller reduces the cost-per-kWh up to 17.5\% compared to a service-greedy First-In-First-Served (FIFS) baseline; the total daily cost decreases more substantially (ranging from 10\% to nearly 30\% reduction) because the proposed controller also avoids economically unattractive charging. These performances demonstrate a highly favorable, tunable Pareto trade-off between economic efficiency and charging service quality for large-scale charging hubs.
\end{abstract}

\end{frontmatter}

\footnotetext{Emails: $\{$\textit{giuseppe.calafiore; luca.ambrosino; matteo.dellarossa}$\}$\textit{@polito.it},\textit{laurent.eg@vinuni.edu.vn}. \\   This document is the results of the research
   project funded by the FAIR - Future Artificial Intelligence Research, with funding from the European Union Next-Generation EU (Piano Nazionale di Ripresa e Resilienza (PNRR) -- Missione 4 Componente 2, Investimento 1.3 -- D.D. 1555 11/10/2022, PE00000013). This manuscript reflects only the authors' views and opinions, neither the European Union nor the European Commission can be considered responsible for them.}
\section{Introduction}

The rapid diffusion of electric vehicles is increasing the operational complexity of charging infrastructure. In large public or workplace charging stations, the charging operator must jointly manage limited grid capacity, uncertain vehicle arrivals and departures, time-varying electricity prices, and the intermittency of on-site renewable generation. When PV production and stationary storage are available, the scheduling problem becomes even richer: local renewable energy can be self-consumed, stored for later use, or replaced by grid energy depending on the prevailing operating conditions. Recent reviews confirm that uncertainty handling, scalability, and the integration of market signals with local energy resources remain central challenges in EV charging control~\cite{Li2024Review,Motlagh2025Review}.

A first large body of literature formulates the charge allocation problem through optimization models, often at the level of individual vehicles. Examples include mixed-integer linear programming (MILP)-based formulations for urban energy systems with EVs, PV, and storage~\cite{Bracco2019,SUN2021125564,BARTOLUCCI2023135426,Nijenhuis25}. These approaches can be very accurate, but they tend to become computationally demanding when the fleet is large or when control decisions must be updated frequently. A second research direction uses optimal-control and model-predictive-control (MPC) techniques to address online operation and uncertainty. Recent contributions include MPC-based EV charging for demand-side management~\cite{Fernandez2024}, adaptive MPC schemes for distribution networks with EV flexibility~\cite{Yang2022AdaptiveMPC}, distributed MPC for multiple charging stations~\cite{Zheng2019OnlineMPC}, and stochastic or uncertainty-aware scheduling strategies for renewable-integrated charging~\cite{CABRERA22,Casini2021PeakMin,Hermans2024, Dutta22}. These works demonstrate the relevance of receding-horizon control, but many of them either rely on vehicle-level models, require detailed information on each charging session, or operate on aggregate daily batches rather than on explicit arrival--charge--departure dynamics.

Our focus in this paper is on charging stations serving many vehicles under high-frequency uncertainty, where one would like to retain a dynamic model rich enough to represent arrivals, charging progress, departures, PV generation, and storage, while keeping the optimization problem tractable. To this end, we build on the idea of aggregating vehicles into classes, or cohorts, defined by their residual charging demand. This cohort-type representation was recently explored in a simpler deterministic setting in~\cite{Calafiore2025_cohort}. Here, we extend such approach along several directions: we consider a full realistic stochastic model for arrivals/departures dynamics, 
we introduce photovoltaic on-site generation and storage dynamics, and we develop a linear programming decision model on a fixed horizon that is robust to statistical model ambiguity. Further, we deploy  such model in a dynamic shrinking-horizon fashion, and present an extensive numerical validation study.

The baseline approach typically implemented in practice for charge management is a \emph{First-In-First-Served} (FIFS) rule.
Under FIFS, charging capacity is allocated as soon as possible to vehicles, according to their order of arrival, without considering variations in electricity prices or the economic value of delaying charging. As a result, a FIFS approach is inherently not cost-aware and represents a service-greedy, price-agnostic extreme: it blindly uses available charging capacity and allocates it according to arrival priority, with no attention to the cost of service.
The  approach proposed here introduces instead controllable flexibility, allowing the operator to move away from this fairness/service extreme in exchange for lower and optimized energy cost.
Our proposed charge controller thus provides a practical way to move along the cost–satisfaction Pareto frontier, relaxing the immediate-service priority of FIFS when doing so creates economically valuable charging flexibility.

\noindent
In detail, the main contributions of the paper are the following.
\begin{itemize}[leftmargin=*]
\item We introduce a class-based state representation in which vehicles are grouped according to the number of charging intervals still required to fulfill their energy demand. This yields a compact model whose dimension depends on the number of classes, not on the number of parked vehicles, and is therefore suitable for large fleets.
\item We develop a discrete-time stochastic model for the coupled arrival--charge--departure process. Vehicle arrivals are modeled through class-wise stochastic arrivals, while departures are represented by class-dependent Bernoulli  mechanisms. The resulting model is coupled with a battery/PV/grid energy-management layer, and we establish basic well-posedness properties such as positivity and integer-valuedness of the stochastic state evolution.
\item We formulate a finite-horizon smart-charging problem that jointly optimizes the class-wise charging decisions and the energy-management variables. The resulting expectation based stage problem is a linear program (LP), which is then deployed in an online shrinking-horizon fashion: at each step, we solve an expected-value optimization problem over the residual planning horizon and, following a discretization/refinement procedure, apply the first control action. Then the horizon is moved one-step forward and the process is repeated.
\item We formally introduce the robust counterpart of our approach, under interval ambiguity on (first moments) of the future parameters. We formally show how feasibility is preserved without leaving the LP framework.
\end{itemize}

An extensive numerical validation campaign is carried out by
considering the combination of three real regional price patterns
with three realistic EV-arrival profiles (hence a grid of nine price/arrivals configurations), and for each configuration we performed dynamic simulations for 100 stochastic scenarios. These experiments highlight the stochastic variability induced by uncertain arrivals, departures, and PV production. The results confirm that the proposed strategy yields substantial economic benefits over FIFS: the saving per delivered kWh is up to 17\% in favorable scenarios, while the total cost-saving ranges from approximately 10\% to 30\%. These gains are achieved at the price of only a very small reduction in service level, confirming a favorable trade-off between operating cost and charging quality. Finally, these results are complemented by an ablation study on a representative configuration, a robust approach to forecast uncertainty, and a parametric sensitivity analysis. 

The paper is organized as follows. Section~\ref{sec:Model} introduces the stochastic charging-station model and the associated expectation dynamics. Section~\ref{Sec:SmartChargingControlProblem} formulates the finite-horizon smart-charging problem and the shrinking-horizon implementation. Section~\ref{sec:uncrob} discusses model ambiguity and a worst-case LP formulation. Section~\ref{sec:NumericalSimulation} presents the numerical study, including a comparison with the batch-MPC benchmark of~\cite{Hermans2024} and  the results of a large-fleet extensive validation campaign over nine price/arrival configurations, in Section~\ref{sec:validation_campaign}. 
An-in depth analysis on a specific one-day scenario is also offered in Section~\ref{sec:one_day_scenario}. Further, a  robust approach example is given in Section~\ref{subsec:robustness_simulations}, and a parametric sensitivity analysis is reported in Section~\ref{sec:sensitivity_analysis}.
Section~\ref{sec:Conclusions} concludes the paper.
\vspace{.2cm}
\textbf{Notation.} We denote by $\N$ the set of nonnegative integers and by $\R$ the set of real numbers. The sets of nonnegative and strictly positive reals are denoted by $\R_{\ge 0}$ and $\R_{>0}$, respectively. Vector and matrix inequalities are understood component-wise. For $k\in\N$, $I_k$ denotes the identity matrix in $\R^{k\times k}$ and $\one\in\R^k$ denotes the all-ones vector.

\section{The model}\label{sec:Model}
In this section we introduce the key  dynamic stochastic model of the
problem under consideration.

\subsection{Variable definitions and assumptions}
We consider a bounded charging station or parking lot where EVs can enter, stay parked while charging or waiting, and depart. Time is discretized with sampling interval $\Delta>0$ and index $t\in\N$; the $t$-th interval is
\[
\cI_t \doteq [t\Delta,(t+1)\Delta).
\]
We assume that arrivals and departures occur only at the boundaries of the sampling intervals.
Each charging socket has the same nominal installed power $P^0>0$, and every vehicle that is selected for charging during $\cI_t$ receives power $P^0$ throughout that interval. A newly arrived vehicle declares an energy request $E$, and we convert it into the number of charging intervals still required by that vehicle as
\[
i=\min\!\left\{\left\lceil \frac{E}{\Delta P^0}\right\rceil,n\right\},
\]
where $n\in\N$ is the maximum admissible number of charging intervals. The integer $i\in\{0,\ldots,n\}$ defines the \emph{class} of the vehicle: class $0$ corresponds to fully charged vehicles, whereas larger values of $i$ correspond to larger residual charging requirements.
For each class $i\in\{0,\ldots,n\}$ and time $t\in\N$ we denote by:
\begin{itemize}[leftmargin=*]
\item $x_i(t)$ the number of class-$i$ vehicles present in the station during $\cI_t$;
\item $c_i(t)$ the number of class-$i$ vehicles that are in charge during $\cI_t$;
\item $a_i(t)$ the number of class-$i$ vehicles that arrive during $\cI_t$; these vehicles will enter the lot only at time $(t+1)\Delta$ and they may thus start their charging cycles from the time-interval $\cI_{t+1}$ onward.
\item $d_i(t)$ the number of class-$i$ vehicles that depart at the end of $\cI_t$; they will actually be outside the lot only at time $(t+1)\Delta$, they may or may not be in charge during the interval $\cI_t$.
\end{itemize}
By definition, all these quantities are integer and nonnegative, and the charging input must satisfy
\[
0\le c_i(t)\le x_i(t),\qquad \forall i\in\{0,\ldots,n\},\ \forall t\in\N.
\]
For convenience we also define the number of non-charging vehicles in class $i$ during $\cI_t$ as
\[
I_i(t)\doteq x_i(t)-c_i(t).
\]

\subsection{The arrival-charge-departures dynamics}\label{subsec_arrival_chargeModel}
We define the aggregate state vector
\[
x(t)\doteq \left[\ba{c}x_0(t)\\ x_1(t)\\ \vdots \\ x_n(t)\ea\right]\in\N^{n+1},
\]
and the arrival and charging vectors
\[
a(t)\doteq [a_0(t),\ldots,a_n(t)]\tran,\;\;c(t)\doteq [c_0(t),\ldots,c_n(t)]\tran.
\]
We adopt the conventions $c_0(t)=0$ and $c_{n+1}(t)=0$ for all~$t$. Define the number of vehicles that remain available for departure after the charging decision by
\[
h_i(t)\doteq x_i(t)-c_i(t)+c_{i+1}(t),\qquad i=0,\ldots,n.
\]
The quantity $h_i(t)$ counts the vehicles that belong to class $i$ at the end of the charging phase during interval $\cI_t$: vehicles already in class $i$ that were not charged remain there, while charged vehicles from class $i+1$ move into class $i$.
The class dynamics are therefore
\begin{equation}
\begin{aligned}
x_i(t+1) &= h_i(t)-d_i(t)+a_i(t)
=x_i(t)-c_i(t)+c_{i+1}(t)-d_i(t)+a_i(t),
\end{aligned}
\label{eq:model1}
\end{equation}
for all $i=0,\ldots,n$.\\
Departures are modeled as a binomial thinning of $h_i(t)$. Let
$\{\mathcal F_t\}_{t\in\mathbb N}$ denote the filtration representing the
information available up to time $t$. After the charging action $c(t)$ has been
selected, $h_i(t)$ is $\mathcal F_t$-measurable, and we assume
\[
d_i(t)\,\big|\,\mathcal F_t \sim \mathrm{Bin}\bigl(h_i(t),\bar\alpha_i(t)\bigr),
\qquad i=0,\ldots,n,
\]
where $\bar\alpha_i(t)\in[0,1]$ is the nominal conditional departure probability
for class $i$ during interval $\mathcal I_t$. Hence,
\[
\mathbb E\!\left[d_i(t)\mid \mathcal F_t\right]=h_i(t)\,\bar\alpha_i(t).
\]
If $h_i(t)>0$, define the realized departure fraction 
$
\alpha_i(t)\doteq {d_i(t)}/{h_i(t)}
$, while for $h_i(t)=0$, set $\alpha_i(t)=0$. Then
\[
\alpha_i(t)\in
\left\{0,\frac{1}{h_i(t)},\ldots,\frac{h_i(t)-1}{h_i(t)},1\right\},
\]
so $\alpha_i(t)$ is discrete-valued. By contrast, $\bar\alpha_i(t)$ is not a
realization but the corresponding conditional mean; indeed, whenever
$h_i(t)>0$,
$
\mathbb E\!\left[\alpha_i(t)\mid \mathcal F_t\right]=\bar\alpha_i(t)
$.
Accordingly, the stochastic state update is
\begin{equation}
x(t+1)=A(t)\,[x(t)+Bc(t)]+a(t),
\label{eq_model0}
\end{equation}
where
\[
A(t)\doteq I_{n+1}-\mathrm{diag}(\alpha(t)),
\;\;
\alpha(t)=[\alpha_0(t),\ldots,\alpha_n(t)]^\top,
\]
\[
\qquad B \doteq
\left[\ba{ccccc}
-1 & 1 & 0 & \cdots & 0 \\
0 & -1 & 1 & \cdots & 0 \\
\vdots & & \ddots & \ddots & \vdots \\
0 & \cdots & 0 & -1 & 1 \\
0 & \cdots & \cdots & 0 & -1
\ea\right].
\]
For prediction, we use the conditional-mean model with respect to the state of
knowledge at the decision time $\tau$. Denoting
$
\mathbb E_\tau[\cdot]\doteq \mathbb E[\cdot\mid \mathcal F_\tau]
$,
the expected dynamics are
\begin{equation}
\bar x(t+1)=\bar A(t)\,[\bar x(t)+Bc(t)]+\bar a(t),\qquad t\ge \tau,
\label{eq_model0_E}
\end{equation}
where
\[
\begin{aligned}
\bar x(t)&\doteq \mathbb E_\tau[x(t)],\quad
\bar a(t)\doteq \mathbb E_\tau[a(t)],\quad
\bar A(t)&\doteq I_{n+1}-\mathrm{diag}(\bar\alpha(t)),
\end{aligned}
\]
with
$
\bar\alpha(t)\doteq
[\bar\alpha_0(t),\ldots,\bar\alpha_n(t)]^\top
$.

\begin{lemma}\label{lemma:Preliminary}
For any $t\in\N$, let $\bar a(t)\in\R_{\ge 0}^{n+1}$ and $\bar\alpha(t)\in[0,1]^{n+1}$. Consider any initial condition $\bar x_0\in\R_{\ge 0}^{n+1}$ and any input sequence $c:\N\to\R^{n+1}$ such that
\[
0\le c(t)\le \bar x(t),\qquad \forall t\in\N.
\]
Then the solution of~\eqref{eq_model0_E} with $\bar x(0)=\bar x_0$ satisfies $\bar x(t)\ge 0$ for all $t\in\N$. The system is also monotonic in $x$, i.e., if $y\geq x\geq c(t)$, then
\begin{equation}\label{eq:monoticity}
 \bar A(t)\,[y+Bc(t)]+\bar a(t)\geq \bar A(t)\,[ x+Bc(t)]+\bar a(t),
\end{equation}
for all $t\in \N$.
Moreover, let $a:\N\to\N^{n+1}$ be any integer-valued arrival sample path, let the departures satisfy the binomial model above, and let $x_0\in\N^{n+1}$. If $c:\N\to\N^{n+1}$ satisfies
\[
0\le c(t)\le x(t),\qquad \forall t\in\N,
\]
then the solution of~\eqref{eq_model0} with $x(0)=x_0$ satisfies $x(t)\in\N^{n+1}$ for all $t\in\N$.
\end{lemma}
\begin{proof}
For the expectation model, the vector $\bar x(t)+Bc(t)$ is component-wise nonnegative because
\[
[\bar x(t)+Bc(t)]_i=\bar x_i(t)-c_i(t)+c_{i+1}(t)\ge 0,
\; i=0,\ldots,n,
\]
by the constraint $0\le c(t)\le \bar x(t)$. Since $\bar A(t)$ has nonnegative diagonal entries and $\bar a(t)\ge 0$, recursion~\eqref{eq_model0_E} implies $\bar x(t+1)\ge 0$ whenever $\bar x(t)\ge 0$; the claim follows by induction.
The claimed monotonicity follows by the fact that $\bar A(t)\geq 0$, for all $t\in \N$.
For the stochastic model, define $h_i(t)=x_i(t)-c_i(t)+c_{i+1}(t)$. If $x(t)\in\N^{n+1}$ and $0\le c(t)\le x(t)$, then $h_i(t)\in\N$ for every $i$. By construction, $d_i(t)\in\{0,1,\ldots,h_i(t)\}$ is integer-valued, hence
\[
x_i(t+1)=h_i(t)-d_i(t)+a_i(t)\in\N.
\]
Because $0\le d_i(t)\le h_i(t)$ and $a_i(t)\ge 0$, we also have $x_i(t+1)\ge 0$. The result follows by induction on $t$.
\end{proof}
In Lemma~\ref{lemma:Preliminary} we established that the arrival--charging--departure model, and its expectation counterpart, are well posed and positive. We next introduce the 
model governing energy production, storage, and release, which will be coupled with model (\ref{eq_model0_E}).

\subsection{Dynamics of energy storage and photovoltaic generation}
We now introduce the energy-management layer associated with the charging station as an  aggregate energy-balance model. Let $s(t)$ denote the energy stored in the local battery during interval $\cI_t$. Its evolution is modeled by
\begin{equation}\label{eq_StorageDynamic}
s(t+1)=s(t)-q(t)\Delta,
\end{equation}
where $q(t)$ is the  battery power exchanged during $\cI_t$: $q(t)>0$ means that energy is withdrawn from storage, while $q(t)<0$ means that energy enters the battery, i.e., that the battery is being charged. The storage state is constrained by
\[
0\le s(t)\le s\ped{max},\qquad \forall t\in\N,
\]
where $s\ped{max}$ is the battery capacity, and the exchange power also satisfies a upper-bound constraint on magnitude:
\[
|q(t)|\le q\ped{max}.
\]
The total charging power required by the EVs during interval $\cI_t$ is
\[
P_c(t)=P^0\sum_{i=0}^n c_i(t)=P^0\one\tran c(t).
\]
This demand is supplied by two controllable power flows: power purchased from the grid, denoted by $p\ped{ntw}(t)$, and power coming from the PV-plus-storage subsystem, denoted by $p\ped{phq}(t)$. Hence,
\[
P_c(t)=p\ped{ntw}(t)+p\ped{phq}(t).
\]
We impose the operational bounds
\beq
0\le p\ped{ntw}(t)\le \Pmax,
\quad
0\le p\ped{phq}(t)\le P\ped{ph}(t)+q(t),
\label{eq:power_budget_ineq}
\eeq
where $\Pmax$ is the contractual power available from the distribution network and $P\ped{ph}(t)\ge 0$ is the PV power available during $\cI_t$. The second inequality expresses the fact that the power injected into the EV chargers from the local energy subsystem cannot exceed the instantaneous PV production plus the battery discharge.

The total power available to the charging station is therefore
\[
P(t)=\Pmax+P\ped{ph}(t)+q(t),
\]
and the corresponding energy quantities over interval $\cI_t$ are obtained by multiplying by $\Delta$; in particular,
\[
E\ped{ntw}(t)=\Delta p\ped{ntw}(t),\qquad E\ped{phq}(t)=\Delta p\ped{phq}(t).
\]
The stationary battery may be charged only from surplus PV generation: grid-imported power can serve EV charging demand, but it is not allowed to charge the battery.

\subsection{Distributions of random quantities}\label{subsec:random_quant}
The proposed controller uses an expectation model, hence it requires forecasts of EVs arrivals, departures, and PV production. The specific stochastic descriptions used for the simulations are the following.
\begin{itemize}[leftmargin=*]

\item \textbf{Vehicle arrivals.} The total number of arrivals is modeled by a non-homogeneous Poisson process with time-varying intensity $\lambda(t)$, a standard choice for EV-arrival modeling in data-driven settings~\cite{AmaraOuali2023NHPP,lee_acndata_2019}. Here $\lambda(t)$ denotes the expected number of arrivals during interval $\cI_t$; if one starts from an hourly arrival-rate profile, the corresponding per-step mean is obtained by multiplying by $\Delta$. Conditioned on the total number of arrivals, vehicles are assigned to classes according to given probabilities $p_i$, with $\sum_{i=0}^n p_i=1$. Hence,
\begin{equation}
\bar a_i(t)=\E\{a_i(t)\}=p_i\,\lambda(t),
\label{eq:arrival_expectation}
\end{equation}
which is generally non-integer.

\item \textbf{Departure matrix used in the deterministic LP.} Given the nominal departure probabilities $\bar\alpha_i(t)$, the deterministic prediction matrix is
\begin{equation}
\bar A(t)=I_{n+1}-\mathrm{diag}(\bar\alpha(t)).
\label{eq:matrixA_expectation_random}
\end{equation}

\item \textbf{Photovoltaic production.} We assume that a nominal clear-sky PV profile $\hat P\ped{ph}(t)$ is available and that actual PV production is obtained by a downward random perturbation,
\[
P\ped{ph}(t)=\hat P\ped{ph}(t)-\xi_t,
\qquad
\xi_t\sim \mathcal U\bigl(0,w\hat P\ped{ph}(t)\bigr),
\]
with $w\in(0,1)$. Therefore,
\begin{equation}
\bar P\ped{ph}(t)=\E\{P\ped{ph}(t)\}=\left(1-\frac{w}{2}\right)\hat P\ped{ph}(t),
\label{eq:PV_production_expectation}
\end{equation}
which provides a conservative expected PV profile.

\end{itemize}

\section{The smart charging control problem}\label{Sec:SmartChargingControlProblem}

The control objective we consider in the smart charging problem is a mixed criterion which accounts for energy costs (to be minimized) and customer satisfaction (to be maximized).
We focus on a \emph{finite-horizon} setup, considering an overall time-span of $[0,\Delta T]$, where $T\in \N$ is the label of the final time interval under consideration, for instance the end of a 24-hour operation period.
In what follows, $\pi_t\geq 0$ denotes the price rate for energy $E\ped{ntw}(t)=\Delta p\ped{ntw}(t)$ drawn from the power network during the  $t$th interval. Electricity prices are assumed to be deterministic because the model operates under a day-ahead market framework, similar to other works \cite{Xu2024}. In this setting, hourly electricity prices for the next day are cleared and published on the day before delivery \cite{ElectricityMarket}, so they are known in advance when the optimization problem is solved. 
The tariff signal $\pi_t$ is therefore  a deterministic input for the model.  Assuming no cost for the self-produced photovoltaic energy, we have that the overall energy economic cost over the considered horizon is
\[
\Upsilon_{0:T-1} = \sum_{t=0}^{T-1}\pi_t \Delta p\ped{ntw}(t).
\]
As a proxy for describing customer satisfaction, we 
use a term that promotes the shifting of vehicles towards lower classes, i.e., classes with lower residual energy needs.
In particular, 
given $\beta_0 \geq  \beta_1 \geq  \beta_2 \geq \cdots\geq  \beta_n\geq 0$ we define
$
\beta\doteq [\beta_0,\ldots,\beta_n]\tran
$ as the \emph{satisfaction weights vector},
and the overall satisfaction of customers is assumed to be expressed by
\begin{equation}\label{eq:Satisfaction}
S_{0:T} \doteq \E \left\{\sum_{t=0}^{T}\beta \tran x(t)\right\}=\sum_{t=0}^{T}\beta \tran \bar x(t)
.
\end{equation}
Such term can be interpreted as a class-weighted occupancy indicator, where lower-class vehicles contribute more strongly, consistent with the assumption that $\beta$ is non-increasing in $i$. From a modeling perspective, this can be regarded as a linear satisfaction measure, introduced to balance the economic-cost term $\Upsilon_{0:T-1}$  defined above. 
Intuitively, a larger value of this term incentivizes a greater number of vehicles to transition to lower residual-demand classes, which is accomplished by applying a more aggressive charging policy.
Many design choices of $\beta$ are possible; in our simulations, we take 
$\beta_i = \rho^i$, with $\rho < 1$ appropriately chosen, 
which formalizes the notion that satisfaction decreases exponentially as the state of charge of the vehicle decreases.

The control objective is naturally a multi-criterion one, since one would like on the one hand to minimize charging costs, and on the other hand to maximize customer satisfaction, and these two goals are naturally conflicting. We here take a classical scalarization approach to the multi-criterion problem, by defining
 a mixed
criterion to be minimized as
\begin{equation}\label{eq:CostFUnctional}
J_{0:T} \doteq    \Upsilon_{0:T-1} - \gamma S_{0:T},
\end{equation}
where $\gamma\geq 0$ is some tunable tradeoff parameter.
The control decisions $(c(t),q(t))$ are given at each instant by a suitable policy $\Theta_t$
\[
(c(t),q(t)) = \Theta_t(\calK_t),
\]
where $\calK_t$ represents the information available at time $t\in \{0,\dots,T-1\}$, which includes the previous  values of the state $x(0),\ldots,x(t)$,  the previous values of the storage $s(0),\ldots, s(t)$, and charging decisions, arrivals, departures, photovoltaic production  at times $0,\dots, t-1$.
For given initial conditions $x(0)\in \N^{n+1}$ and $s(0)=s_0\in [0,s\ped{max}]$, 
the smart charging control problem would amount to minimizing over the policies $\Theta_0,\ldots,\Theta_{T-1}$ the expected cost $J_{0:T}$, while satisfying the system dynamics and  constraints (in expectation). 
This stochastic dynamic-programming formulation is, however, computationally intractable in general because the policies are infinite-dimensional objects.

We hence here consider a classical approximate solution approach 
inspired by model predictive control (MPC), in which we solve repeatedly the control problem in open loop at each stage $\tau=0,\ldots,T-1$, with fixed decisions, over a receding and shrinking horizon, and take as effective control decision the first element from the result of each stage. We formally describe the procedure in the following subsection.\\

\subsection{Receding-horizon control procedure}\label{subsec:RecedingHorizonProcedure}
In this subsection, we formally describe the implementation of our control policy, which aims to minimize the overall cost defined in~\eqref{eq:CostFUnctional} within a receding-horizon iterative framework.

To this end, we introduce a stage (or timer) variable, denoted by $\tau \in \{0, \ldots, T-1\}$.
At each stage $\tau$ the current integer-valued state $x(\tau)$ and storage $s(\tau)$ are available together with the current photovoltaic production $P\ped{ph}(\tau)$, and we thus use them, as well as the other information $\calK_\tau$, to update
the forward expected  arrivals and departures rates $\bar a(t)$ and $\bar \alpha(t)$ for $t\geq  \tau$, and expected photovoltaic production $\bar P\ped{ph}(t)$ for $t>\tau$.  
For notational convenience, in what follows, we suppose that $\bar P\ped{ph}(\tau)= P\ped{ph}(\tau)$, since, as said, the current photovoltaic production is known.
We then organize the problem variables in matrix/vector format as:
\[
\begin{aligned}
C_\tau 
&\doteq [c(\tau)\,\cdots\,c(T-1)]\in\Real{n+1,T-\tau},
\end{aligned}
\]
\[
\begin{aligned}
X_\tau 
&\doteq [\bar x(\tau)\,\cdots\,\bar x(T)]\in\Real{n+1,T+1-\tau}, \\
Q_\tau 
&\doteq  [q(\tau)\,\cdots\,q(T-1)]\in\Real{1,T-\tau},\\
S_\tau 
&\doteq  [\bar s(\tau)\,\cdots\,\bar s(T)]\in\Real{1,T+1-\tau}, \\
P_{\text{ntw},\tau} 
&\doteq  [p_{\text{ntw}}(\tau)\,\cdots\,p_{\text{ntw}}(T-1)]\in\Real{1,T-\tau},\\
P_{\text{phq},\tau} 
&\doteq  [p_{\text{phq}}(\tau)\,\cdots\,p_{\text{phq}}(T-1)]\in\Real{1,T-\tau},
\end{aligned}
\]
 resulting in a total of $(2n+6)(T-\tau)+n+2$ variables. Then,  the \emph{stage-$\tau$ smart charging problem} is cast in the form of the following \emph{linear program}:
\bea
\label{eq:mainLPmodel}
&\bar p_\tau =&\min_{C_\tau,X_\tau,Q_\tau,S_\tau,P\ped{ntw,\tau},P\ped{phq,\tau}}    \sum_{t=\tau}^{T-1}\pi_t \Delta p\ped{ntw}(t)\\&\;\;\;&\;\;\; - \gamma \sum_{t=\tau}^{T}\beta \tran \bar x(t) \nonumber \\
& \mbox{s.t.:} & 
P^0 \one\tran c(t)  = p\ped{ntw}(t) + p\ped{phq}(t) ,\quad \forall \,t\in \calT_\tau, \nonumber \\
&& 0\leq p\ped{ntw}(t) \leq  \Pmax , \quad\forall\, t\in \calT_\tau, \nonumber \\
&& 0\leq p\ped{phq}(t)\leq \bar P\ped{ph}(t) + q(t),\quad \forall \,t\in \calT_\tau, \nonumber \\
 && 0\leq c(t) \leq \bar x(t),\quad \forall\, t\in \calT_\tau , \nonumber\\
  &&  c_0(t)=0,\quad \forall\, t\in \calT_\tau , \nonumber\\
 && \bar x(t+1) = \bar A(t)\left [\bar x(t) +B c(t)\right] + \bar a(t), \;\;\forall \,t\in \calT_\tau, \nonumber \\
  &&    \bar x(\tau)=x(\tau), \nonumber\\
 && \bar s(t+1) = \bar s(t) - q(t) \Delta, \quad\forall\; t\in \calT_\tau  , \\
 && \bar s(\tau)=s(\tau), \nonumber\\
&&   0\leq \bar s(t+1)\leq s\ped{max},\quad \forall t\in \calT_\tau, \nonumber\\
&&  |q(t)| \leq q\ped{max}, \quad \forall\;t\in \calT_\tau,\nonumber
\eea
where $\calT_\tau \doteq \{\tau,\ldots,T-1\}$.
Note that the final positivity constraint $\bar x(T)\geq 0$ is not explicitly stated, since it is trivially satisfied, as proven in Lemma~\ref{lemma:Preliminary}.

Let us denote an optimal solution of the above optimization problem \eqref{eq:mainLPmodel}  by $C^\star_\tau,X^\star_\tau$, $Q^\star_\tau,S^\star_\tau$, $P^\star\ped{ntw,\tau},P^\star\ped{phq,\tau}$. In particular, this provides a real-valued optimal charging schedule $c^\star(t)$ for $t\in \{\tau, \ldots, T-1\}$ for the vehicles classes, as well as
the optimal storage plan $q^\star(t)$, for $t\in \{\tau,\ldots,T-1\}$. 
The current optimal charging policy $c^\star(\tau)$, however, is not exactly implementable in practice, since it is real valued; we further know that $0 \leq c^\star(\tau) \leq \bar x(\tau)=x(\tau)$  because of the constraints in problem \eqref{eq:mainLPmodel}. Hence, we discretize $c^\star(\tau)$ into an integer-valued solution by flooring  each of its entries.
 More formally, we define such integer vector by
\begin{equation}\label{eq:integerPolicy}
\wt c(\tau)\doteq\lfloor c^\star(\tau) \rfloor\in \N^{n+1}
\end{equation}
where $\lfloor\cdot \rfloor$ denotes the component-wise floor function.

%
To satisfy the equality constraint 
\[
P^0 \one\tran \wt c(\tau)  = p\ped{ntw}(\tau) + p\ped{phq}(\tau)
\]
we possibly need to decrease the value of $p^\star\ped{ntw}(\tau)$ (and possibly of $p^\star\ped{phq}(\tau)$). Operationally, one may first reduce grid power and, only if needed, reduce the local-supply setpoint consistently with the corresponding storage usage. 
More formally, we set

\begin{subequations}\label{eq:DecreasedEnergy}
  \begin{equation}  \label{eq:PntwMigliorato}
   \wt p\ped{ntw}(\tau)\doteq\max\left \{ P^0 \one\tran \wt c(\tau)  -  p^\star\ped{phq}(\tau),\, 0\right \},
   \end{equation}
   \begin{equation}\label{eq:PphqMigliorato}
     \wt p\ped{phq}(\tau)\doteq P^0 \one\tran \wt c(\tau)  -  \wt p\ped{ntw}(\tau),
   \end{equation}
\end{subequations}
Similarly, due to the flooring of the charging decision variable $c(t)$ as defined in~\eqref{eq:integerPolicy}, we might be able to use less energy from the storage systems. Thus, more formally, we can define
\begin{equation}\label{eq:DedreasedCharge}
\begin{aligned}
\wt q(\tau) \doteq \max \Big\{ 
    \min \Big\{ q^\star(\tau),\, \wt p\ped{phq}(\tau) - \bar P\ped{ph}(\tau) \Big\}, 
    -q\ped{max},\; \frac{\bar s(\tau) - s\ped{max}}{\Delta}
\Big\}.
\end{aligned}
\end{equation}
Following this ``reassignment'' of such decision variables (as performed in~\eqref{eq:integerPolicy},~\eqref{eq:DecreasedEnergy},~\eqref{eq:DedreasedCharge}), it can be proven that  such refined solution satisfies the inequalities at time \( \tau \), and enables the computation of the subsequent initial conditions, namely \( x(\tau+1) \) and \( s(\tau+1) \), in compliance with the constraints. This result is formally stated in the following statement.
\begin{lemma}\label{lemma:DiscretizationSteo}
Given $\tau\in \{0,\dots, T-1\}$, $x(\tau)\in \N^{n+1}$ and $s(\tau)\in [0,s\ped{max}]$, consider  an optimal solution of the optimization problem \eqref{eq:mainLPmodel} denoted  by $C^\star_\tau,X^\star_\tau$, $Q^\star_\tau,S^\star_\tau$, $P^\star\ped{ntw,\tau},P^\star\ped{phq,\tau}$. Consider also $\wt c(\tau)\in \N^{n+1}$,  $\wt p\ped{ntw}(\tau)$, $\wt p\ped{phq}(\tau)$, $\wt q(\tau)\in \R$ defined in~\eqref{eq:integerPolicy},~\eqref{eq:DecreasedEnergy},~\eqref{eq:DedreasedCharge} and  realizations at time $\tau$ of the arrival/departures random processes, denoted by $\wt a(\tau)\in \N^{n+1}$ and $\wt d(\tau)\in \N^{n+1}$.
It holds that:
\begin{itemize}[leftmargin=*]
\item The following constraints are satisfied:
\[
\begin{aligned}
&P^0 \one\tran \wt c(\tau)  = \wt p\ped{ntw}(\tau) + \wt p\ped{phq}(\tau) ,\\
 &0\leq \wt p\ped{ntw}(\tau) \leq  \Pmax, \quad 
 0\leq \wt p\ped{phq}(\tau)\leq \bar P\ped{ph}(\tau) + \wt q(\tau),\\
 &0\leq \wt c(\tau) \leq  x(\tau), \quad \;|\wt q(\tau)| \leq q\ped{max}, 
 \end{aligned}
\]
\item The resulting next-step class and charge states satisfy
\[
\begin{aligned}
x(\tau+1)&=x(\tau)+B\wt c(\tau)+\wt a(\tau)-\wt d(\tau)\in \N^{n+1},\\s(\tau+1)&= s(\tau) - \wt q(\tau) \Delta\in [0,s\ped{max}].\\
\end{aligned}
\]
\end{itemize}
\end{lemma}
The proof of this statement is given in Appendix~A, to avoid breaking the flow of the presentation.

Summarizing, the ``refined''  policy $\wt c(\tau)$,  $\wt p\ped{ntw}(\tau)$, $\wt p\ped{phq}(\tau)$, $\wt q(\tau)$  defined in~\eqref{eq:integerPolicy},~\eqref{eq:DecreasedEnergy},~\eqref{eq:DedreasedCharge}  is our heuristic control, to be applied at time $\tau$. Lemma~\ref{lemma:DiscretizationSteo} assures us that such control strategy provides a feasible trajectory, and allow us to compute the next (integer valued) initial state $x(\tau+1)$ and the state-of-charge of the charging system, i.e., $s(\tau+1)$. 

 At time step $\tau+1$, given such states $x(\tau+1)\in \N^{n+1}$ and $s(\tau+1)\in [0,s\ped{max}]$, we can observe the current photovoltaic production $P\ped{ph}(\tau+1)$, which allow us to initialize again the procedure, considering the stage-$(\tau+1)$ charging problem. We can then solve the corresponding LP program, to obtain, after the refinement defined in~\eqref{eq:integerPolicy},~\eqref{eq:DecreasedEnergy},~\eqref{eq:DedreasedCharge}, the policy at time step $\tau+1$; then, we iterate the process up to time $T-1$ to obtain the global control policy.\\
Summarizing, by carrying out this iterative procedure for $\tau = 0,\dots,T-1$, we obtain a heuristic policy $ \wt c(\tau), \wt q(\tau)$ which is constructed so as to satisfy the constraints, to  be implementable in an online fashion, and to minimize the expected cost over the remaining time interval $[\tau, T]$.

\subsection{Micromanagement at the vehicle's level}

Even after discretization, the schedule \( \tilde{c}(t) \) defined in~\eqref{eq:integerPolicy} does not specify \emph{which individual vehicles}
should be charged. This practical decision is made at a micro-management scale based on a deterministic precedence rule.
Upon arrival, each vehicle is assigned a unique ordinal identification number (ID),
so that \( k = 1, 2, \ldots \) represents its order of entry into the parking or charging facility. At time \( t \), for each class \( i\in \{0,\dots, n\} \), we:
\begin{enumerate}[leftmargin=*]
    \item collect the identifiers of all \( x_i(t) \) vehicles currently present in the lot into the set \( K_i(t) \); by construction, \( \tilde{c}_i(t) \leq x_i(t) \).
    \item determine the subset of vehicles to be charged according to a first-in-class rule:
    select the \( \tilde{c}_i(t) \) vehicles corresponding to the smallest identifiers in \( K_i(t) \),
    that is, those with the earliest arrival times.
\end{enumerate}

This approach guarantees that the integer charging decisions \( \tilde{c}(t) \) are faithfully implemented
at the individual vehicle level, preserving consistency with the optimal class-level decisions
while maintaining fairness and operational continuity within each class.

This micromanagement approach is duly applied to all 
control policies we simulate in Section~\ref{sec:NumericalSimulation}, and permits to obtain and evaluate all vehicle-level statistics
and performance indicators.


\section{Robustness to model ambiguity}
\label{sec:uncrob}
The proposed smart-charging model is stochastic: arrivals, departures and PV generation are random processes. However, the control action is computed at each step by solving problem~\eqref{eq:mainLPmodel}, which is based
on the expected values of the involved stochastic quantities.
When such expected values are imprecisely known, e.g., they are known only to belong to given intervals of confidence, we adopt the viewpoint that the controller is designed to be robust with respect to such model ambiguity. This viewpoint is consistent with standard robust-optimization formulations based on uncertainty sets for model parameters~\cite{BenTal2009,Bertsimas2011}.

 Since the expectation dynamics are affine in these parameters and the constraints are monotone with respect to them, solving~\eqref{eq:mainLPmodel} with conservative expected values yields a policy that remains feasible for all distributions whose moments lie within the specified ranges. In this sense, the LP~\eqref{eq:mainLPmodel} can be interpreted as a robust expectation model: expected quantities are employed, but they are chosen so as to hedge against misspecification of the stochastic model, without introducing probabilistic constraints and without altering the linear structure of the optimization problem.

From a technical standpoint, in order to formalize this degree of flexibility within our LP-based design framework, we study a tailored \emph{relaxation} of problem~\eqref{eq:mainLPmodel}, in which the state-recursion equality constraints are replaced by inequality constraints. We then establish that this relaxation is actually tight, i.e., that there is no gap between the original problem and its relaxed formulation. This is the objective of the following subsection;  in the subsequent one we use this technical result to formally present the robust version of our control strategy.
 
\subsection{Inequality-form optimization problem}
\label{sec-ineq}
In this subsection we consider again problem~\eqref{eq:mainLPmodel} and replace the equality constraints of the state recursion by inequalities.  For given $\tau\in \{0,\dots, T-1\}$ and $x(\tau)\in \N^{n+1}$, $s(\tau)\in [0,s\ped{max}]$, let us consider the ``relaxed'' problem 
\bea
\label{eq:mainLPmodel_ineq}
&\bar p_{rel,\tau} \hspace{-0.1cm}=&\hspace{-0.25cm}\min_{C_\tau,X_\tau,Q_\tau,S_\tau,P\ped{ntw,\tau},P\ped{phq,\tau}}    \sum_{t=\tau}^{T-1}\pi_t \Delta p\ped{ntw}(t) - \gamma \sum_{t=\tau}^{T}\beta \tran \bar x(t) \nonumber  \\
& \mbox{s.t.:}& 
P^0 \one\tran c(t)  = p\ped{ntw}(t) + p\ped{phq}(t) ,\quad \forall \,t\in \calT_\tau, \nonumber \\
&& 0\leq p\ped{ntw}(t) \leq  \Pmax , \quad\,\forall\, t\in \calT_\tau, \nonumber \\
&& 0\leq p\ped{phq}(t)\leq \bar P\ped{ph}(t) + q(t),\quad \forall \,t\in \calT_\tau, \nonumber \\
 && 0\leq c(t) \leq \bar x(t),\quad  t\in \calT_\tau, \nonumber \\
   &&  c_0(t)=0,\quad \forall\, t\in \calT_\tau , \nonumber\\
 && \bar x(t+1) \leq   \bar A(t)\left [\bar x(t) +B c(t)\right] + \bar a(t) , \;\;\forall \,t\in \calT_\tau, \nonumber \\
  &&    \bar x(\tau)\leq x(\tau), \nonumber\\
 && \bar s(t+1) =  \bar s(t) - q(t) \Delta, \;\forall\; t\in \calT_\tau  ,  \\
 && \bar s(\tau)=  s(\tau), \nonumber\\
&&  0\leq \bar s(t+1)\leq s\ped{max},\quad \forall t\in \calT_\tau, \nonumber\\
&&  |q(t)| \leq q\ped{max}, \; \forall\;t\in \calT_\tau.\nonumber
\eea
In this formulation, the equality constraints in~\eqref{eq:mainLPmodel} concerning the evolution of the state $x$ are replaced by inequality constraints. 
The following proposition holds, see the Appendix~B for a detailed proof.

\begin{proposition}\label{prop_ineqversion}
 For a given time $\tau\in \{0,\ldots,T-1\}$, initial conditions $x(\tau)\in \N^{n+1}$, $s(\tau)\in [0,s\ped{max}]$, let $C^\star\in \R^{n+1,T-\tau}$,  $Q^\star$,  $P^\star\ped{ntw}$, $P^\star\ped{phq}\in \R^{1,T-\tau}$,  and $X^\star\in \R^{n+1,T+1-\tau}$, $S^\star\in\R^{1,T+1-\tau}$ be an optimal solution\footnote{Since the time $\tau \in \{0,\ldots,T-1\}$ is fixed, for notational simplicity in the statement and the subsequent proof of Proposition~\ref{prop_ineqversion} we drop the $\tau$ subscript from the optimal matrices.} for problem \eqref{eq:mainLPmodel_ineq}, and let
$p^\star_{rel,\tau}$ be the corresponding optimal value. Let $p^\star_{\tau}$ be the optimal value of the original problem~\eqref{eq:mainLPmodel}.\\ Then the following hold:
\begin{enumerate}
\item If  $\gamma\geq 0$ and $\beta_i\geq 0$ for all $i$, then  $p^\star_{rel,\tau}=p^\star_\tau$.
\item  If
$\gamma>0$ and $\beta_i>0$ for all $i$, then $C^\star,Q^\star,P^\star\ped{ntw}$, $P^\star\ped{phq}, X^\star, S^\star$ is optimal also for~\eqref{eq:mainLPmodel}.
  \end{enumerate}
\end{proposition}

Roughly speaking, Proposition~\ref{prop_ineqversion} shows that whenever $\gamma,\beta \ge 0$, we can solve the inequality formulation~\eqref{eq:mainLPmodel_ineq} and obtain an optimal objective value $p^\star_{rel,\tau}$ that coincides with the optimal value $p^\star_\tau$ of the original equality-constrained problem~\eqref{eq:mainLPmodel}. Moreover, when $\gamma,\beta>0$, any optimal solution of~\eqref{eq:mainLPmodel_ineq} is necessarily  optimal also for~\eqref{eq:mainLPmodel}.

The main advantage of working with the inequality formulation~\eqref{eq:mainLPmodel_ineq} is that convex inequality constraints preserve convexity when enforced in a robust (i.e., worst-case) sense, as discussed in the next subsection.

\subsection{Robust LP implementation}
\label{subsec:robustness_ambiguity}
In this subsection, we assume that expected values of the stochastic processes underlying the control problem are known only within prescribed confidence intervals.
More precisely, the instantaneous expected power generated by the photovoltaic system, denoted by $\bar P\ped{ph}(t)$, is described through the interval uncertainty model
\[
\bar P\ped{ph}(t) \in[  \bar P\ped{ph,inf}(t),  \bar P\ped{ph,sup}(t)],  \quad \forall\, t\in\{0,\ldots,T-1\},
\]
where $0\leq \bar P\ped{ph,inf}(t)\leq   \bar P\ped{ph,sup}(t)$ for every $t\in \{0,\dots ,T-1\}$.
Similarly, average vehicle arrivals are expressed as 
\[
\bar a_i(t) \in [\bar a_{i,\text{inf}}(t),\, \bar a_{i,\text{sup}}(t)],  \quad \begin{aligned}\forall\, t\in\{0,\ldots,T-1\},\quad\forall \, i\in \{0,\ldots,n\},\end{aligned}
\]
where  $0\leq \bar a_{i,\text{inf}}(t)\leq \bar a_{i,\text{sup}}(t)$.
Uncertainty in the nominal departure probabilities $\bar \alpha(t)$ entering $\bar A(t) \doteq I - \mbox{diag}(\bar \alpha(t))$ is dealt with again via an interval model of the form
\[
\bar \alpha_i(t)  \in [\bar \alpha_{i,\text{inf}}(t),\, \bar \alpha_{i,\text{sup}}(t)],  \quad \begin{aligned} \forall\, t\in\{0,\ldots,T-1\},\quad
\forall\,i\in\{0,\ldots,n\},
\end{aligned}
\]
where $ 0\leq \bar \alpha_{i,\text{inf}}(t)\leq  \bar \alpha_{i,\text{sup}}(t)\leq 1$.
Accordingly, let us introduce the notation
\[
\bar A_{\text{inf}}(t) \doteq I - \text{diag}(\bar \alpha_{\text{sup}}(t))\leq \bar A_{\text{sup}}(t) \doteq I - \text{diag}(\bar \alpha_{\text{inf}}(t)), 
\]
for all $t\in \{0,\dots, T-1\}$, where the $\text{inf}$ and $\text{sup}$ labels are interchanged so as to preserve the component-wise ordering. 
In the robust approach we shall minimize the  objective function of problem \eqref{eq:mainLPmodel_ineq}, while ensuring that the constraints remain feasible for all possible realizations of the uncertainties. 

In the aforementioned interval model, this amounts to consider the worst-case constraint inequalities, that is, the inequalities that induce the smallest (i.e., minimal with respect to the inclusion partial relation) feasible set. Summarizing, the arising ``worst case'' LP reads:
\bea
\label{eq:mainLPmodelWORSTCASE}
\hspace{-0.25cm} &\bar p_{worst,\tau}&\hspace{-0.25cm} =\hspace{-0.25cm}\min_{\substack{C_\tau,X_\tau,Q_\tau,\\S_\tau,P\ped{ntw,\tau},P\ped{phq,\tau}}}
 \sum_{t=\tau}^{T-1} \pi_{t} \Delta p\ped{ntw}(t) - \gamma \sum_{t=\tau}^{T}\beta \tran \bar x(t)  \nonumber \\
& \mbox{s.t.:} &  
P^0 \one\tran c(t)  =  p\ped{ntw}(t) + p\ped{phq}(t) ,\quad \forall \,t\in \calT_\tau, \nonumber, \\
&& 0\leq p\ped{ntw}(t) \leq  \Pmax , \quad t\in \calT_\tau, \nonumber \\
&& 0\leq p\ped{phq}(t)\leq \bar P\ped{ph, \text{inf}}(t) + q(t),\quad \forall \,t\in \calT_\tau, \nonumber \\
 && 0\leq c(t) \leq \bar x(t),\quad  t\in \calT_\tau, \nonumber \\
   &&  c_0(t)=0,\quad \forall\, t\in \calT_\tau , \nonumber\\
 && \bar x(t+1) \leq  \bar A\ped{inf}(t) \left [\bar x(t)+B c(t)\right] + \bar a_{\text{inf}}(t), \nonumber \\
  &&    \bar x(\tau)\leq x(\tau), \nonumber\\
 && \bar s(t+1) =  \bar s(t) - q(t) \Delta, \;\forall\; t\in \calT_\tau  ,  \\
 && \bar s(\tau)= s(\tau), \nonumber\\
&&  0\leq \bar s(t+1)\leq s\ped{max},\quad \forall\, t\in \calT_\tau, \nonumber\\
&&  |q(t)| \leq q\ped{max}, \quad \forall\,t\in \calT_\tau,\nonumber
\eea
Suppose that, for each $\tau \in \{0, \dots, T\}$, we solve the LP problem~\eqref{eq:mainLPmodelWORSTCASE} instead of the ``true'' problem~\eqref{eq:mainLPmodel_ineq}, which would require knowledge of the exact $\bar P\ped{ph}(t), \bar a_i(t), \bar \alpha_i(t)$ -- information that is unavailable in the ``ambiguous'' setting considered in this section. It can be easily proven, via~Lemma~\ref{lemma:DiscretizationSteo} and Proposition~\ref{prop_ineqversion}, that any optimal solution for problem~\eqref{eq:mainLPmodelWORSTCASE} is feasible also for the problem~\eqref{eq:mainLPmodel_ineq}, for any occurrence of the statistics 
\[
\begin{aligned}
& \bar P\ped{ph}(t) \in [  \bar P\ped{ph,inf}(t),  \bar P\ped{ph,sup}(t)],\\\bar a_i(t) \in [&\bar a_{i,\text{inf}}(t),\, \bar a_{i,\text{sup}}(t)], \;\;\bar \alpha_i(t)  \in [\bar \alpha_{i,\text{inf}}(t),\, \bar \alpha_{i,\text{sup}}(t)], 
\end{aligned}
\]
for all  $t\in\{0,\ldots,T-1\}$, for all $i\in \{0,\dots n\}$.   
Moreover, along any realization governed by such (unknown) statistics, we can implement the corresponding policy $c(\tau),q(\tau)$ at step $\tau$, after discretizing $c(\tau)$ and refining the other variables, as described in Subsection~\ref{subsec:RecedingHorizonProcedure}. 

Summarizing, the overall robust control policy (whether considered under the nominal model in expectation or when implemented via a realization-driven shrinking-horizon procedure) is guaranteed to remain feasible
regardless of the actual expectation statistics, as long as they stay within their confidence intervals. 
The robust solution is, in principle, suboptimal with respect to 
a truly optimal solution computed assuming perfect advance knowledge of the process expectations. However, since 
such advance knowledge is not realistic, the optimal solution provides in practice a safer control policy, guaranteed to be feasible for a range of possible process expectations.

The qualitative differences of a robust control policy with respect to the nominal one is illustrated via numerical experiments in Section~\ref{subsec:robustness_simulations}.
%

\section{Numerical simulations}\label{sec:NumericalSimulation}

We next present an extensive campaign of numerical experiments,  to illustrate the behavior and validate the proposed shrinking-horizon smart charging framework. First, we 
discuss comparatively our approach with 
 the batch-type model predictive controller (MPC) proposed by Hermans et al. in~\cite{Hermans2024}.
 Then, we present the main validation experiments 
via large-fleet simulations involving nine configurations of price/arrival-intensity profiles, each run over 100 realizations of the stochastic parameters in the model. For these simulations, we
compare the performances of our proposed model against  a baseline First‑In‑First‑Served (FIFS) strategy, which is a standard and widely adopted benchmark in the EV charging literature,  commonly used to assess smart charging algorithms~\cite{Calafiore2025, Liu2021, Madaram2024}.
All the numerical simulations, figures and tables in the paper are reproducible using the code in the reproducibility package available in the  \href{https://github.com/beppe969/SmartCharge}{\texttt{SmartCharge\_code} } repository.

\subsection{Batch-Type MPC for EV~Charging:  Hermans et al. (2024)}\label{sec:Batch-Type MPC}

A relevant benchmark for aggregate-level EV charging control is the batch-type model predictive controller (MPC) proposed by Hermans et al. in~\cite{Hermans2024}.
In this paper, the authors consider a real-world office microgrid located in the Netherlands with 174 chargers, PV generation, and a 1~MW grid connection limit. Since individual vehicle data are unavailable, the daily EV charging demand is modeled as a single batch energy \(E_{\mathrm{evs}}^{(d)}\), forecasted from historical occupancy patterns. Their MPC, formulated as a quadratic program with 15-minute sampling, optimizes the aggregate charging power \(P_{\mathrm{evs}}(k)\) over a dynamic horizon to meet two energy targets: a minimum energy \(E_{T_1}\) by an intermediate checkpoint \(T_1\) (e.g., 16:00) and the full requested energy \(E_{T_2}\) by \(T_2\) (e.g., 23:00), while reducing the daily grid peak power \(\lvert P_{\mathrm{grid}}^{\mathrm{peak}} \rvert\). The controller enforces power balance \(P_{\mathrm{grid}}(k)=P_{\mathrm{evs}}(k)-P_{\mathrm{pv}}(k)\), grid limits, and path constraints on \(P_{\mathrm{evs}}\) derived from forecasted connected vehicles and clipping thresholds \cite{Hermans2024}.

\begin{figure}[t!] 
    \centering
    \includegraphics[width=0.49\textwidth]{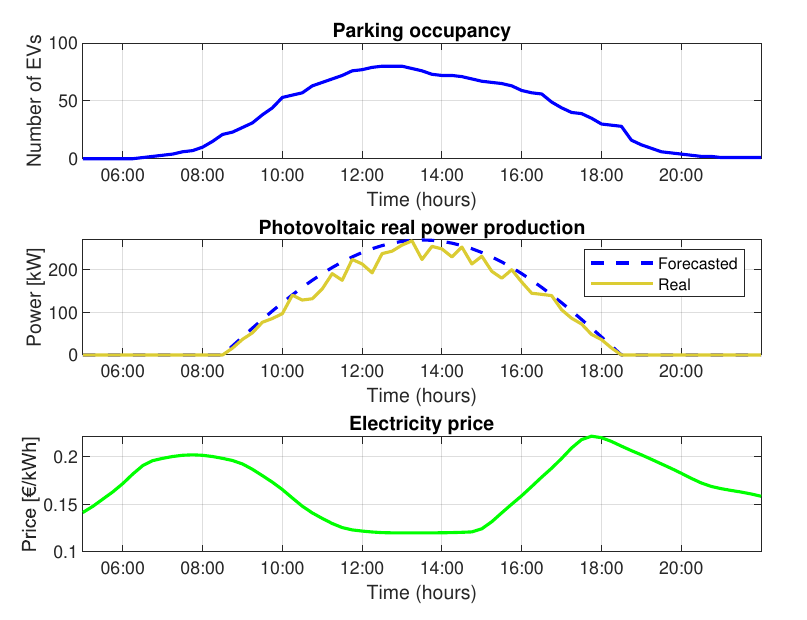} 
    \caption{Input data for simulation inspired by Hermans et al~\cite{Hermans2024}. Top plot: number of EVs inside the charging station system during the day; middle plot: forecasted and actual PV power production from 8 a.m. to 6 p.m.; bottom plot: electricity price pattern on February 13th, 2023, in the Netherlands.}
    \label{fig:hermans_setting}
\end{figure}
\subsubsection{Numerical comparison}
Although our class-based model is designed to address large-scale EV fleet scenarios and has a different optimization objective with respect to~\cite{Hermans2024}, we aimed to make a fair comparison with this approach by recreating a similar system setting. We analyzed 89 EVs entering the charging station system over the course of one day, distributed as shown in the top plot of Figure \ref{fig:hermans_setting}, which closely resembles the pattern considered by~\cite{Hermans2024} in Figure 9c. The photovoltaic power generation throughout the day, as depicted in the middle plot of Figure \ref{fig:hermans_setting}, occurs between 8 AM and 6 PM, aligning with the data presented in Figure 11b of~\cite{Hermans2024}. The bottom plot illustrates the electricity price used for our  comparison, which reflects the actual price (sourced from~\cite{Italy_data}, which collects the day-ahead electricity price data for European countries) in the Netherlands on February 13th, 2023, that is, the same day \cite{Hermans2024} considered for  their control strategy.
Since no storage system was considered in \cite{Hermans2024}, we accordingly set $s\ped{max}=q\ped{max}=0$ in our implementation. This also emphasizes that, although storage represents an important feature of our model, it can be easily excluded thanks to the modular structure of the proposed LP formulation, setting the corresponding parameters to zero.

The simulations were conducted using a $15$-minute time step, following the same discretization methodology as in~\cite{Hermans2024}. For our model, we set the number of classes to $n = 17$, balancing computational granularity with the need to effectively represent small-scale parking scenarios such as the one studied in~\cite{Hermans2024}, where approximately $80$ vehicles visited the parking facility during the day under analysis. The other physical parameters were chosen accordingly with the values reported in~\cite{Hermans2024}. The satisfaction component of the cost function in \eqref{eq:Satisfaction} is defined by setting $\beta_i = 0.5^i$ (see Appendix C for more details), and the trade-off parameter $\gamma$ was set to $0.5$ following a calibration procedure aimed at
balancing 
cost reduction and service quality.
Figure~\ref{fig:hermans_simulated} reproduces the qualitative behavior of the batch-MPC benchmark reported in Figure~9c of~\cite{Hermans2024}: the charging power is split between PV energy and grid-purchased energy over the course of the day. Figure~\ref{fig:hermans_comparison} compares the power drawn from the grid by Hermans' controller and by our class-based model.

 The two controllers optimize different objectives: Hermans' batch MPC directly targets peak shaving and therefore keeps the grid profile flatter, whereas our controller is cost-oriented and can concentrate purchases in cheaper periods, while also accounting for EV user satisfaction, which is not considered by the Batch model. Consistently with this interpretation, Table~\ref{tab:Hermans} shows that Hermans' controller attains a lower grid peak, while the proposed controller achieves a lower average purchase price per kWh from the grid, in line with its cost-focused objective. Because the two controllers rely on different objectives and aggregate information structures, this comparison should be interpreted as a behavioral benchmark on a small-fleet scenario rather than as a claim of strict dominance. The primary test bed for the proposed approach remains a large-fleet setting, where class aggregation is most advantageous. This setting is examined in detail in the following sections.



\begin{figure}[t!]
    \centering
    \begin{subfigure}{0.48\textwidth}
        \centering
        \includegraphics[width=\textwidth]{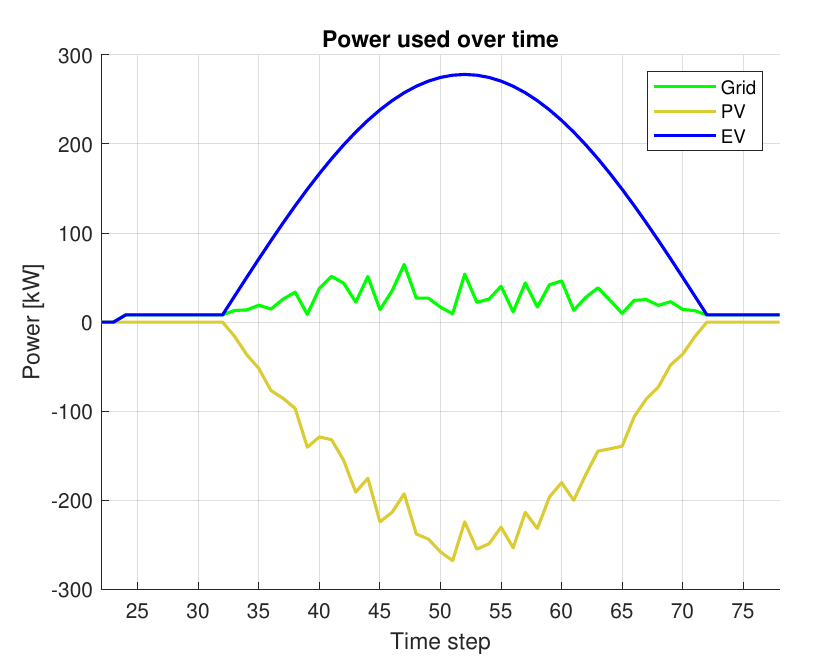}
        \caption{Reproduction of Hermans MPC: power used to charge EVs comes from grid or PV production.}
        \label{fig:hermans_simulated}
    \end{subfigure}
    \hfill
    \begin{subfigure}{0.48\textwidth}
        \centering
        \includegraphics[width=\textwidth]{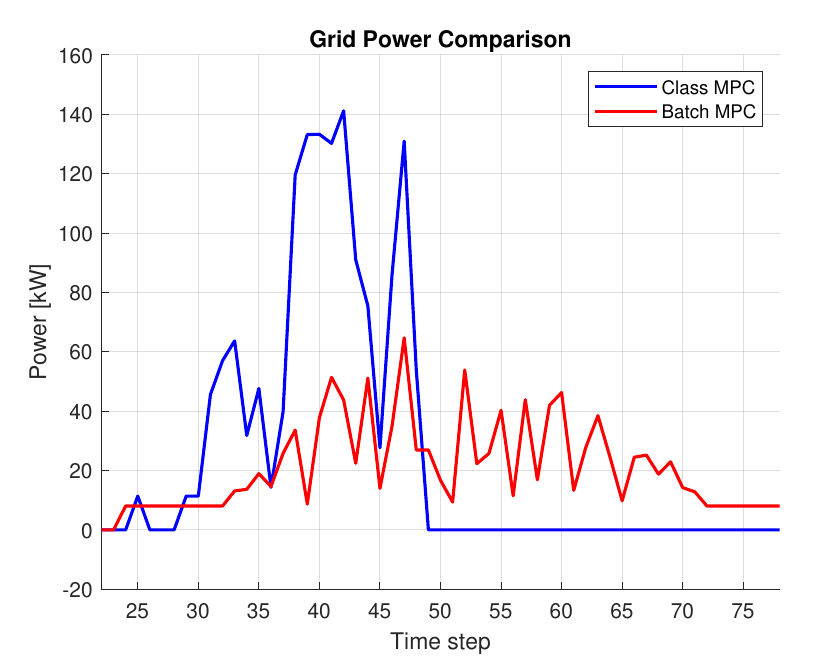}
        \caption{Comparison of grid power between Hermans MPC (red) and our class MPC (blue); Hermans' MPC aims at peak shaving, ours at cost-saving.}
        \label{fig:hermans_comparison}
    \end{subfigure}
    \caption{Comparison between Hermans MPC and our class MPC.}
    \label{fig:hermans_sidebyside}
\end{figure}

\begin{table}[ht]
    \centering 
       \caption{Output comparison between Hermans' Batch MPC and our proposed Class MPC.}
    \begin{tabular}{@{}lccc@{}}
        \toprule
        Indicator & Unit       & Batch & Class \\ \midrule
        Number of vehicles        & [ ]             & 89         & 89         \\
        Total energy given        & [kWh]           & 1846.8    & 1846.8    \\
        Grid power peak           & [kW]            & 64.7       & 141.2      \\
       Purchased energy     & [kWh]           & 316.25     & 364.12     \\
        Grid purchased {total} cost       & [\euro]         & 48.53      & 54.91      \\
        Cost per {purchased} kWh              & [\euro/kWh]     & 0.1534      & 0.1508      \\
        \bottomrule
    \end{tabular}
     \label{tab:Hermans}
\end{table}

\subsection{First-In-First-Served (FIFS) charging strategy with PV-First energy management}\label{sec:FIFS}
We next describe in more detail the FIFS method that will be used 
as a baseline for performance comparison.
The FIFS strategy implements a simple yet effective heuristic for scheduling electric vehicle~(EV) charging based on arrival order and class progression. Vehicles are assigned to charging slots in the order of arrival, and their charging process progresses through predefined classes that reflect increasing states of charge. The simulation maintains a queue $\mathcal Q(t)$ of connected vehicles per class, and at each time step $t$, vehicles may depart stochastically based on class-specific departure probabilities $\bar\alpha_{k(v)}(t)$, where $k(v)$ is the class of the $v$-th EV. Departures are simulated using a Bernoulli sampling process, and class updates ensure that vehicles progress toward lower-index classes once charged in their current class.

New vehicle arrivals are added to their corresponding class queue, and the charging schedule is updated as follows: priority is given to vehicles already being charged (continued charging), and only when additional power is available new vehicles are assigned to free charging slots, respecting physical power limits. The number of slots assigned is determined by the remaining power budget divided by the nominal socket power. Class advancement occurs when vehicles complete a charge cycle, moving from class $i$ to $i-1$, until reaching class 0.

The energy management strategy follows a PV-first policy. At each time step, the power $P_c(t)$ required for charging is computed, and the system allocates resources in a prioritized order. First, the algorithm uses available photovoltaic (PV) power to satisfy as much of the charging demand as possible. If the PV generation is insufficient, the storage unit is discharged to cover the remaining deficit, subject to discharge power limits and current state of charge. If a gap still remains, the system draws energy from the power grid, up to a maximum contracted limit. Any unused PV energy at that time step is considered for storage charging, again subject to the maximum power exchange rate and capacity bounds. This approach is entirely deterministic with respect to power allocation, relying only on stochasticity for modeling arrivals and departures. It respects all hardware constraints like grid limit, PV availability, storage capacity and power limits. However, it does not optimize cost or customer satisfaction explicitly, and charging slot allocation is solely based on availability and arrival time, without anticipating future arrivals or energy price variations. The FIFS algorithm with PV and storage handling is summarized in Appendix D.

The FIFS policy should be interpreted as a service-greedy benchmark. For a fixed realization of the system uncertainty and a fixed sequence of available charging capacities, FIFS is work-conserving and lexicographically optimal with respect to the cumulative charging service delivered to vehicles ordered by arrival time. Indeed, whenever a feasible schedule charges a later vehicle while an earlier unfinished vehicle is available, exchanging that charging quantum in favor of the earlier vehicle preserves feasibility and improves the arrival-order service vector. In this sense, FIFS represents an extreme operating point that prioritizes immediate service and arrival-order fairness while being completely agnostic to electricity prices and future PV availability. Our  proposed approach introduces instead scheduling flexibility by allowing a controlled reduction of this immediate-service priority in exchange for lower energy cost.

\subsection{Configurations setting}\label{Sec:configuration_setting}
The large-fleet case study considers a charging hub with approximately 200 parking spaces equipped with $22$~kW AC charge points, of which up to 100 can draw power from the grid simultaneously, corresponding to $\Pmax = 2200$~kW. The selected station size is intended to represent a large charging hub rather than a typical public charging site. Real-world deployments and planning studies show that EV charging hubs can comprise more than 100 charging ports \cite{Eleport2025LargestHubs, Heath2024}, and rely on co-located battery energy storage system and PV to mitigate grid connection limits and peak demand \cite{Fresia2025, IEAPVPS2025ChargingStations}. For this reason, the adopted battery capacity and network power limit should be interpreted as representative of a high-demand charging node with constrained grid interconnection, where storage is used to buffer demand peaks and smooth the interaction with the upstream grid \cite{Rehman2024}. For the purposes of this paper, we assume that the charging station has sufficiently large parking capacity, so all arriving vehicles are admitted in the parking.
%
Our dynamic approach is evaluated across multiple  configurations, each characterized by distinct features. For an extensive simulation campaign, we consider three different electricity price curves combined with three daily vehicle-arrival patterns, all modeled as non-homogeneous Poisson processes but with different shapes.

Two of the electricity price profiles are obtained from real data collected in Italy in 2025 \cite{Italy_data}, using the average hourly prices recorded in {June} and {December} and resampling them to match the time-step length $\Delta$. These two months correspond to the summer and winter solstices, are six months apart, and, as shown in Figure~\ref{fig:tariffs}, exhibit markedly different price profiles in terms of shifted peaks and valleys. The third electricity price profile follows a different, piecewise-constant structure: it is derived from the {Vietnam} Electricity business-customer retail tariff \cite{EVNTariff2025}, using the published off-peak, standard, and peak rates together with EVN's time-of-use period definitions \cite{EVNTOU2016}. 
The official half-hour boundaries are mapped onto our $\Delta$-time-step discretization, and the conversion rate of 30000~VND = 1~\euro     is adopted to enable comparison with the other two price profiles. These three price profiles are concisely denoted, in what follows, by \emph{June},~\emph{December} and~\emph{Vietnam}.

\begin{figure}[t!]
    \centering

    \begin{subfigure}{0.49\textwidth}
        \centering
        \includegraphics[width=\linewidth]{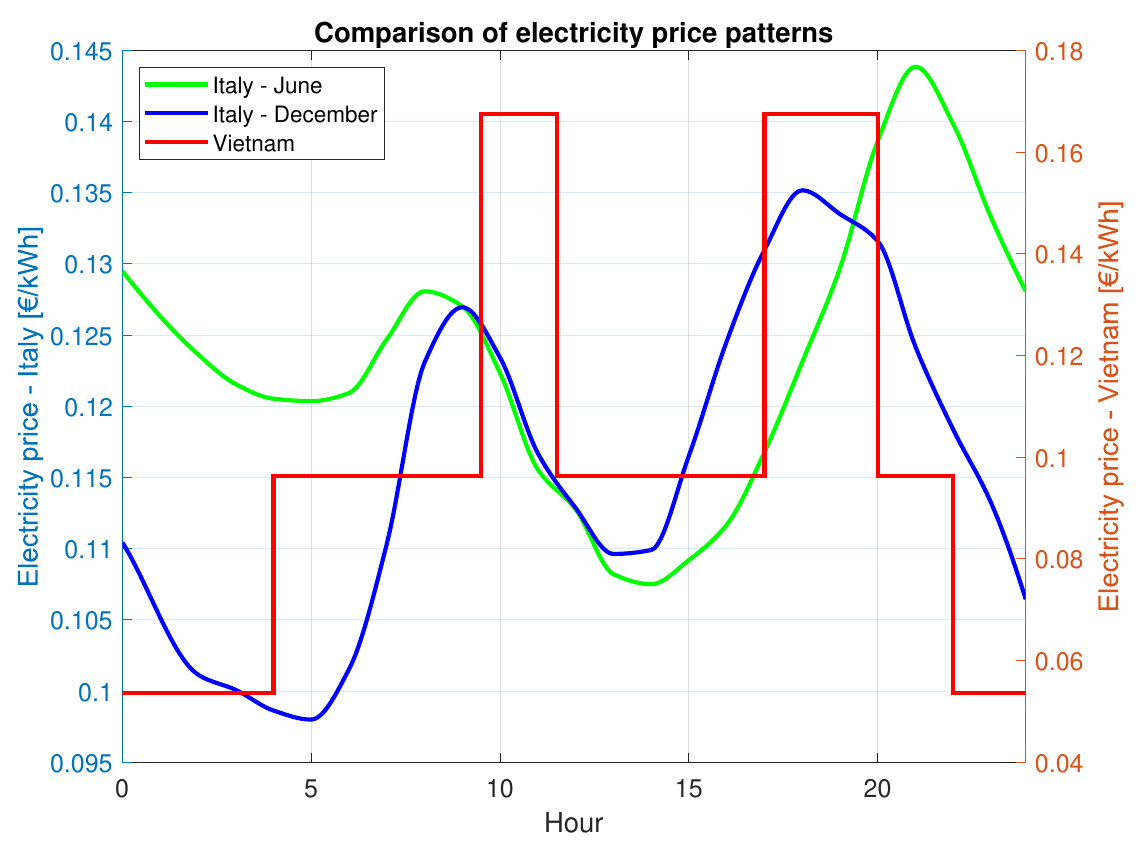}
        \caption{Electricity price patterns considered: averaged prices from Italy (June and December 2025) and Vietnam tariff.}
        \label{fig:tariffs}
    \end{subfigure}
    \hfill
    \begin{subfigure}{0.49\textwidth}
        \centering
        \includegraphics[width=\linewidth]{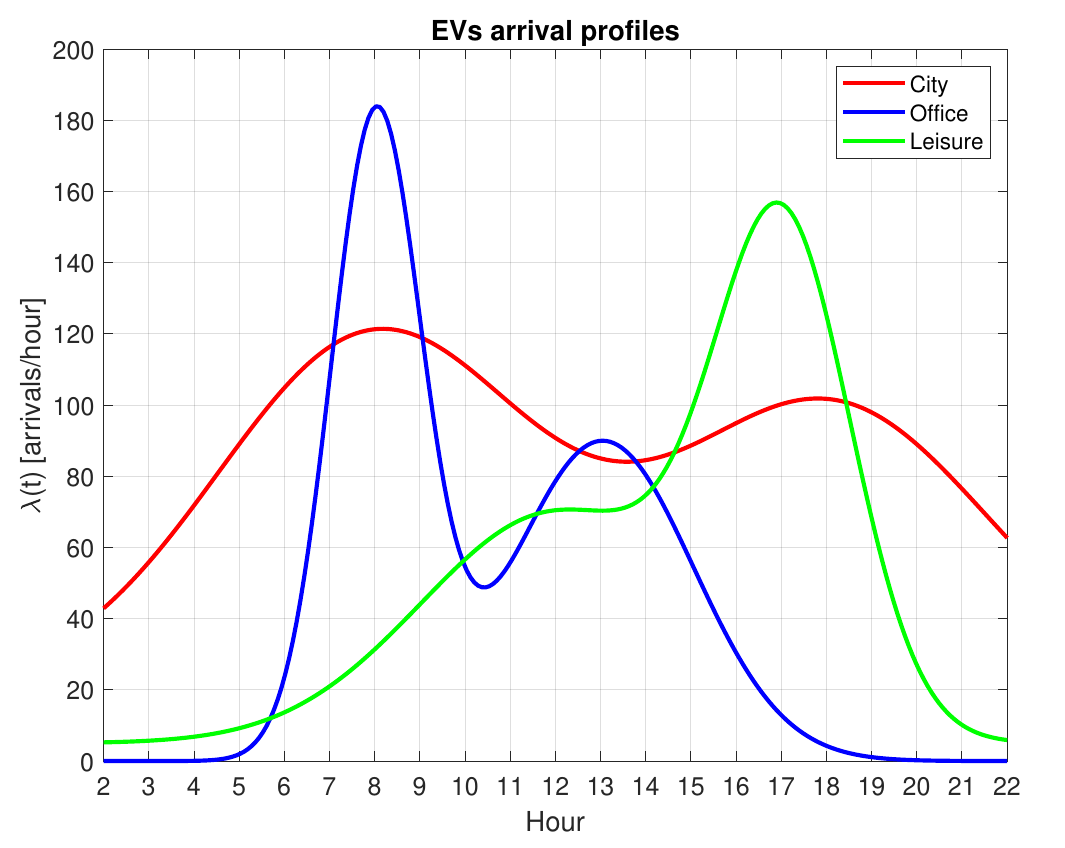}
        \caption{Poisson arrival profiles for office, city and leisure EV charging demand.}
        \label{fig:arrivals_profile}
    \end{subfigure}

    \caption{Left: electricity price patterns. Right: EV arrival profiles.}
\end{figure}

The three EV arrival patterns considered are illustrated in Figure~\ref{fig:arrivals_profile} and represent three common scenarios. The first pattern models an \emph{Office} private parking lot, where nighttime arrivals are zero and most EVs arrive in the early morning as commuters head to work; a second, smaller peak is also included around lunchtime to capture part-time workers. The second pattern represents a \emph{City}-downtown parking lot, with a lower peak than the office case but with EV arrivals spread more evenly throughout the day, including off-peak hours. The last pattern describes a \emph{Leisure} parking lot, for instance near a large amusement park, where arrivals are concentrated in the afternoon or after working hours. Furthermore, and since we are considering a time horizon of 24 hours, we always assume the station's gate to be closed for entrance before 2 a.m. and after 10 p.m. This feature is formally obtained by setting the arrival rates $\lambda(t) = 0$ at both these initial and final time intervals, consistently across all three scenarios in Figure~\ref{fig:arrivals_profile}.
The profile is first specified in vehicles per hour and then converted into the per-step Poisson mean used in the discrete-time simulations, again denoted by $\lambda(t)$ for simplicity. EVs are then divided into classes depending on a real-world calibrated per-class probability $p_i$ (see Appendix C). These three arrivals patterns are concisely denoted, in what follows, by \emph{Office},~\emph{City} and~\emph{Leisure}, respectively.

The resulting simulation setting, with 9 different combinations of  electricity prices (\emph{June, December}, and \emph{Vietnam}) and EVs arrival patterns (\emph{Office, City}, and \emph{Leisure}), allows our model to be fairly tested across multiple realistic configurations, showing its potential to adapt to several contexts. 

\subsubsection{Parameters setting}\label{sec:Parameters_setting}

To instantiate the shrinking-horizon controller in a realistic large-fleet scenario, we specify the required parameters.
We detail here the selection of the control design parameters, namely the overall time horizon $[0,T]$, the number of classes $n$ and the discretization step $\Delta$. To properly define these variables, we first need to fix the physical parameter representing the instantaneous maximum charging power of each charging socket, denoted by $P^0$. We chose $P^0 = 22\,\mathrm{kW}$, as justified in Appendix C. 
Given this, we set $T$, $n$, $\Delta$ as follows:
\begin{itemize}[leftmargin=*]
    \item \textbf{Time horizon $T$:} the simulation horizon is set to be $T=\frac{24 \text{hours}}{\Delta}$, modeling the whole evolution of the charging station during a working day.
    \item \textbf{Number of classes $n+1$:} a representative mid-size battery electric vehicle typically features a usable battery capacity on the order of 60 kWh, which lies within the range reported for many contemporary compact and mid-size models \cite{EVDatabase_Capacity_2024}. For such a vehicle, a 22 kW AC charge point provides a reasonable benchmark: under conditions with approximately constant power, a full charge from empty to full would take about three hours, consistent with values reported in technical guides and consumer-oriented charging references \cite{PodPoint_ChargingTime_2025}. In practice, tapering and onboard-charger limits may increase this time, but three hours remains a realistic order of magnitude for modeling purposes. In a time-discretized optimization framework with time step $\Delta$ (in hours), achieving a charging horizon of roughly three hours corresponds to using about $n \approx 3/\Delta$ classes in our model.

    \item \textbf{Time step $\Delta$:}  the length of the simulation step is crucial for the computational time, and it determine each class duration. We evaluated the average simulation runtime over 100 realizations of problem \eqref{eq:mainLPmodel} at the first step ($\tau = 0$), and the corresponding results are reported in Table \ref{tab:computational_time}. We then chose $\Delta = 5$ minutes as an effective tradeoff between runtime and time granularity. The corresponding simulation horizon (24 h) yields $T=288$. Accordingly, in the numerical study we use $n=30$, that is, 31 classes including class $0$.

\end{itemize}

The values of the remaining model parameters are summarized in Table~\ref{tab:parameters}, while a detailed discussion of these choices is deferred to Appendix C to preserve the flow of the presentation.

\begin{table}[ht!b]
    \centering
    \caption{Comparison of computational time for different choices of $n$, $\Delta$ and $T$, evaluated across 100 simulations.}
    \label{tab:computational_time}
    \renewcommand{\arraystretch}{0.95}
    \setlength{\tabcolsep}{2.1pt}
    \begin{tabular}{|c|c|c|c|c|c|}
        \hline
        \textbf{$n+1$} & \textbf{$\Delta$} & \textbf{T} &
        \textbf{LP variables} & \textbf{Max runtime} & \textbf{Avg runtime} \\
        \hline
        4 & 60 min & 24 &
        293 & 0.49 s & 0.29 s \\
        \hline
        6 & 30 min & 48 &
        775 & 0.50 s & 0.36 s \\
        \hline
        11 & 15 min & 96 &
        2508 & 0.73 s & 0.57 s \\
        \hline
        16 & 10 min & 144 &
        5201 & 0.89 s & 0.80 s \\
        \hline
        31 & 5 min & 288 &
        19040 & 2.79 s & 2.06 s \\
        \hline
        51 & 3 min & 480 &
        50932 & 7.18 s & 5.68 s \\
        \hline
    \end{tabular}
\end{table}

\begin{table*}[t!]
\caption{Summary of model parameters in Subsection~\ref{sec:Parameters_setting} and Appendix C.} 
\centering
\renewcommand{\arraystretch}{1.3} 
\begin{adjustbox}{width=0.97\textwidth}
\begin{tabular}{|l|c|c||l|c|c|}
\hline
\textbf{Name} & \textbf{Symbol} & \textbf{Value} &
\textbf{Name} & \textbf{Symbol} & \textbf{Value} \\
\hline
Time step & $\Delta$ & $5$ minutes &
Time horizon & $T$ & $\frac{24}{\Delta} = 288$ steps \\
\hline
Max charging power & $P^0$ & $22$ kW &
Number of classes & $n+1$ & $31$ (with $n=30$) \\
\hline
Price & $\pi_i$ & \{\emph{June},\emph{December},\emph{Vietnam}\} & Max network power & $\Pmax$ & 2.2 MW 
 \\
\hline
Arrival rate & $\lambda(t)$ & \{\emph{Office},\emph{City},\emph{Leisure}\} &
Entry class probability & $p_i$ & Based on \cite{Shenzen_data} \\
\hline
Base departure rate & $\alpha_i^{\mathrm{base}}$ & $ \frac{n+1-i}{n+2}$ &
Departure rate & $\alpha_i(t)$ & $\alpha_i^{\mathrm{base}} \times 0.15$ \\
\hline
Storage capacity & $s_{\mathrm{max}}$ & $2.2$ MWh &
Max storage exchange & $q_{\mathrm{max}}$ & $1.1$ MW\\
\hline
PV profile  & $\hat P_{ph}(t)$ & Sinusoidal (peak at 12:30)&
PV realization & $P_{ph}(t)$ & $\hat P_{ph}(t) - \xi_t$, \hspace{0.1cm}$\xi_t \sim U(0, 0.75\hat P_{ph}(t))$ \\
\hline
Satisfaction weights & $\beta_i$ & $2^{-i}$ &
Tradeoff parameter & $\gamma$ & Properly calibrated \\
\hline
Initial EV state & $x_i(0)$ & $0$ for all $i$ &
Initial storage & $s(0)$ & $0$ kWh \\
\hline
\end{tabular}
\end{adjustbox}
\label{tab:parameters}
\end{table*}

\subsubsection{Tradeoff parameter $\gamma$: Pareto analysis}\label{sec:Pareto analysis}

The tradeoff parameter $\gamma$, balancing the focus of the optimization model between cost-saving and customer satisfaction, must be chosen considering the order of magnitude of the two terms of the objective function in \eqref{eq:mainLPmodel}, but also depending on the user preference, since bigger values for $\gamma$ lead to more customers-aware/satisfaction-oriented solutions. 

To make the tradeoff apparent, we  computed the Pareto frontiers
for each of the nine price/arrivals configurations that we discussed previously.
This computation was performed by solving  the fixed horizon problem~\eqref{eq:mainLPmodel} at $\tau=0$, for several grid values of $\gamma$, all the other parameters being held fixed.
Then, for each of the nine problems and for each optimal solution for a given $\gamma$, we plotted a point on the corresponding Pareto curve, by putting on the horizontal axis the average relative class improvement per vehicle (computed as arrival class minus departure class, divided by arrival class) and on the vertical axis the average cost per delivered MWh. The result is the nine Pareto curves shown in Figure~\ref{fig:frontiere}. 
Next, we used these curves to calibrate a suitable $\gamma$ value for each configuration, to be used in the later actual dynamic 
validation campaign in Section \ref{sec:validation_campaign}.
To this purpose, we simply selected the smallest $\gamma$ that ensures at least 50\% average relative delta class. For example, 
for the \emph{December/Office} configuration shown in more detail 
in the bottom panel of Figure \ref{fig:frontiere}, the selected $\gamma$ value was 4.4.The selected $\gamma$ values for the other configurations are reported in Table~\ref{tab:optimal_gamma} in Appendix C.

It is worth to observe that selecting a different $\gamma$ for each configuration is natural for the purpose of comparing the different configurations via simulations, and practically meaningful. Indeed, a charging facility managers in a specific location (e.g., Vietnam) know their specific price profile and expected arrival profile, and they obviously shall tune the controller (i.e., the $\gamma$ value) so to reach the expected target (e.g., 50\% average relative delta class). For this reason, the validation simulations have been performed at ``equal target'' among the nine configurations, and not at ``equal $\gamma$''.
\begin{figure}[t!]
    \centering
    \begin{subfigure}{0.49\textwidth}
        \centering
        \includegraphics[scale=0.43]{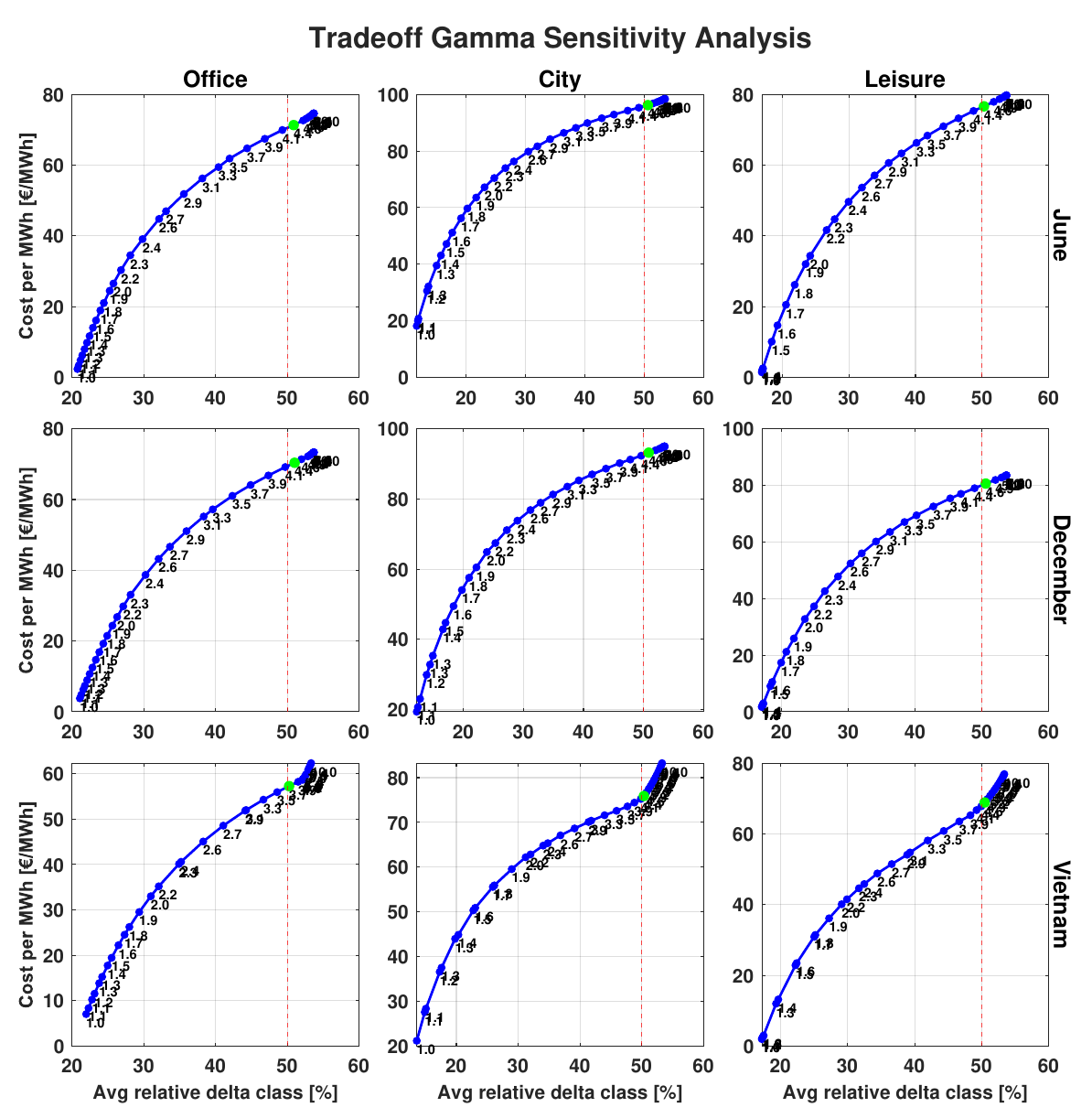}
    \end{subfigure}
    \hfill
    \begin{subfigure}{0.49\textwidth}
        \centering
        \includegraphics[scale=0.48]{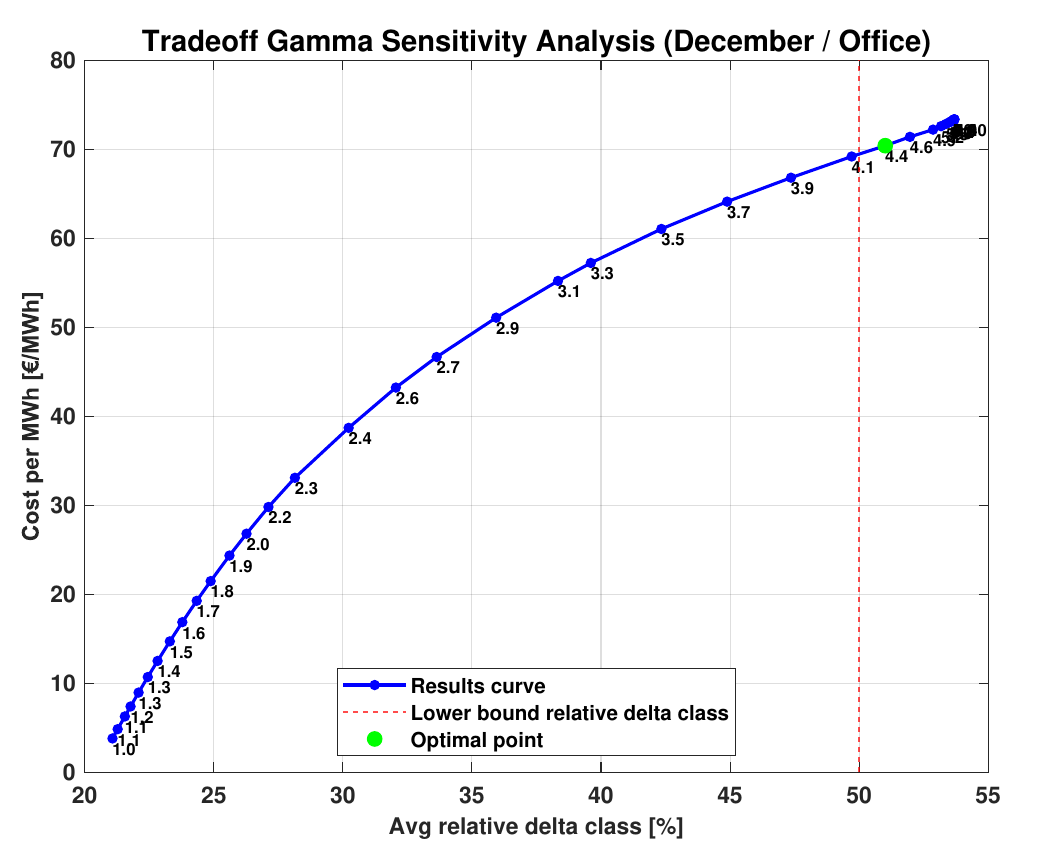} 
    \end{subfigure}
    \caption{Top panel: $\gamma$-calibration results combining the prices \{June, December, Vietnam\} on the rows with the arrivals \{Office, City, Leisure\} on the columns. Bottom panel: zoom into one specific configuration (\emph{December/Office}), selected  $\gamma = 4.4$ shown as a green dot.}
    \label{fig:frontiere}
\end{figure}

\subsection{Statistical validation campaign }\label{sec:validation_campaign}
We next present an extensive dynamic simulation campaign for validation of our proposed approach and comparison with the FIFS baseline. For each of the nine price/arrivals configuration previously described, we simulate 100 independent one-day scenarios with different realizations of the uncertain factors such as PV power production and vehicles arrivals and departures. The arrival processes are identical for both models; however, since departures depend on the class reached under each policy, hence they are control-dependent, the exact number of EVs inside the charging station at any time $\tau$ may differ between our proposed approach and FIFS.
For each configuration and scenario we evaluate our controller deployed in shrinking-horizon and we compare its performance with respect to the FIFS policy. 

\subsubsection{Performance indicators}

The first group of indicators concerns the economic performance of the charging strategy. The key quantity in this respect is the \emph{average price per kWh} (or MWh) delivered.
This metric captures the nominal cost of energy delivery for both the proposed method and the for FIFS baseline, and a relative indicator can also be derived by expressing the optimal controller's cost in relation to the FIFS baseline cost. In this way we can normalize results across different scenarios and better isolate the effects of operational or informational constraints on system performance.
Further cost-related indicators include the total cost over an operation day, for both the proposed controller and the FIFS baseline, and the cost saving, defined as the difference of the previous two quantities.

The second group of metrics, as the objective function of Equation \eqref{eq:mainLPmodel} suggests, is devoted to the quality of service experienced by the EV users. We do not limit our analysis to the computation of the chosen ``proxy'' measure of user satisfaction, i.e., the function $S$ defined in~\eqref{eq:Satisfaction}. Instead, we introduce additional performance indicators to more comprehensively assess both the fairness and the perceived quality of the proposed charging schedules from the users’ perspective.
The main indicator in this group is the \emph{relative delta class}, which is defined, for each EV, as the ratio between the charging cycles completed and the entry class; this measure  is then averaged over all EVs. Accordingly, a value of 100\% indicates that all vehicles have been fully charged, whereas 0\% corresponds to the absence of any charging activity. Additional complementary measures include the percentage of fully charged EVs (i.e. EVs leaving the charging station being in class 0), the percentage of EVs that depart with more than two-thirds of the battery charged (meaning in class $n/3$ or lower) and the waiting time steps before start charging. 

The third group of indicators includes the metrics that characterize the behavior of the power and energy system, which is a useful information even though not explicitly optimized by our controller. These quantities include the total energy delivered during the whole day, divided into the energy provided by the PV and storage system and the total energy purchased from the grid; the maximum power drained from the network at anytime indicates the level of stress on the grid, while the instantaneous maximum and total absolute energy exchanged by the battery over the day indicate how much and how the storage system has been used. These metrics are important to assess both the feasibility and the operational stress imposed by the charging policy on the underlying infrastructure.

\subsubsection{Validation campaign results }\label{sec:validation_campaign_results}

One hundred stochastic realizations are simulated for each of the nine configurations introduced earlier, and for each of them the optimal tradeoff coefficient $\gamma$ (recall the definition of the cost in~\eqref{eq:CostFUnctional}) is selected as described in Figure \ref{fig:100simulations}. Since the actual number of EVs, the departure realizations, and the PV generation profile vary from one scenario to another, the resulting performance exhibits a significant stochastic spread, which is captured by boxplots. Figure \ref{fig:box_savingperkWh} and Figure \ref{fig:box_relative_delta_class} report, respectively, the cost saving per delivered kWh and the relative delta-class difference between our model and FIFS, both expressed in percentage points, while Figure \ref{fig:scatter_performances} shows every simulation point, highlighting the trade-off between these two quantities across all configurations and showing the full cloud of results rather than only aggregated statistics. As usual, the boxplots summarize the variability over the 100 simulations, with the central box containing the middle 50\% of the data and the line inside each box indicating the median.

Overall, the most favorable cost performance is obtained by the \emph{Vietnam/Office} configuration, which yields a 17.5\% cost-per-kWh savings, the largest among all cases. This is consistent with the scatter plot, where the same configuration tends to occupy the upper-left region, indicating high savings but also a relatively larger service distance compared to FIFS. By contrast, the \emph{City} arrival profile generally leads to the lowest and less variable savings, with a 4\% average in the worst case with the \emph{December} price pattern, whereas the \emph{Office} profile consistently offers a good compromise between economic benefit and customer service. The relative delta-class difference remains negative in all configurations, confirming the FIFS approach to be a suitable upper-bound benchmark for this metric; however, for almost every configuration, this degradation is limited and typically remains close to zero, with an average lower than -3\% for five configurations out of nine. This is particularly relevant because FIFS is already a strong benchmark in terms of user satisfaction, as it does not account for peak shaving or cost minimization and simply aims to charge EVs as early and as much as possible. Therefore, a small reduction in service quality is acceptable when it enables the substantial cost savings obtained by our shrinking horizon optimization model. One of the advantages of the proposed method is that it allows the operator to explicitly balance monetary cost savings against customer satisfaction, while FIFS offers no comparable trade-off mechanism because it follows a fixed heuristic policy.

\begin{figure}[t!]
    \centering
    \begin{subfigure}{0.48\textwidth}
        \centering
        \includegraphics[scale=0.42]{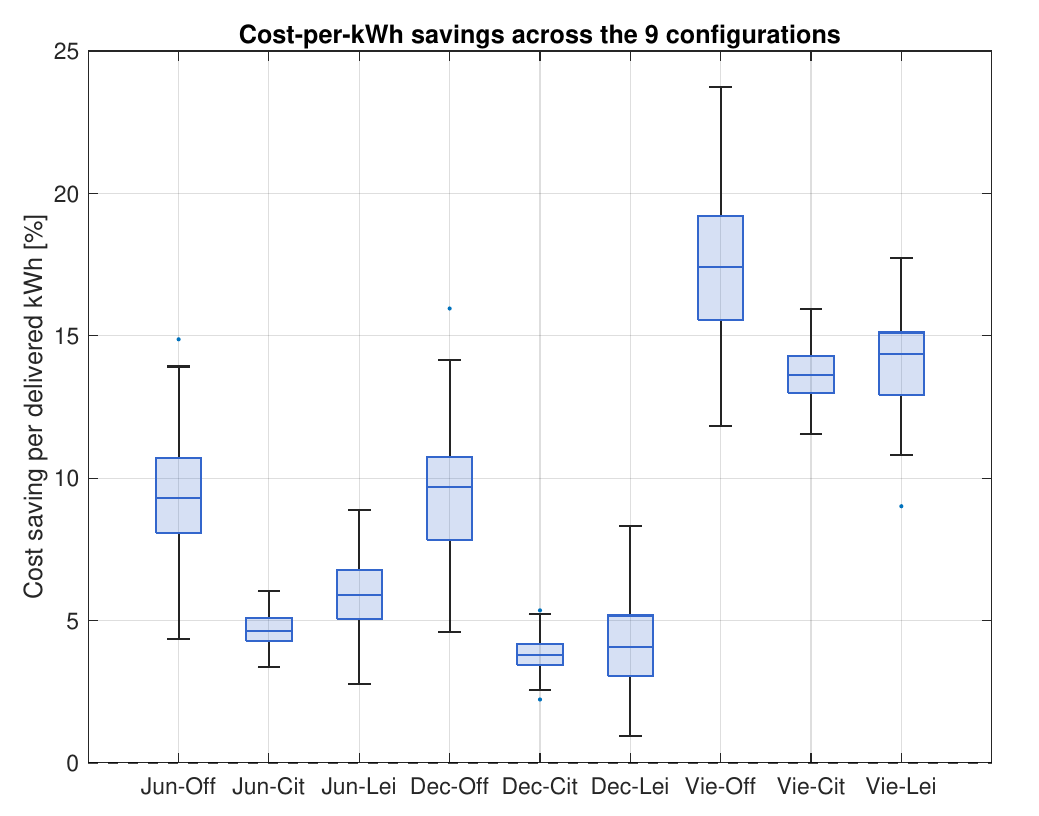}
        \caption{100-simulations cost-per-kWh savings across nine configurations.}
        \label{fig:box_savingperkWh}
    \end{subfigure}
    \hfill
    \begin{subfigure}{0.48\textwidth}
        \centering
        \includegraphics[scale=0.42]{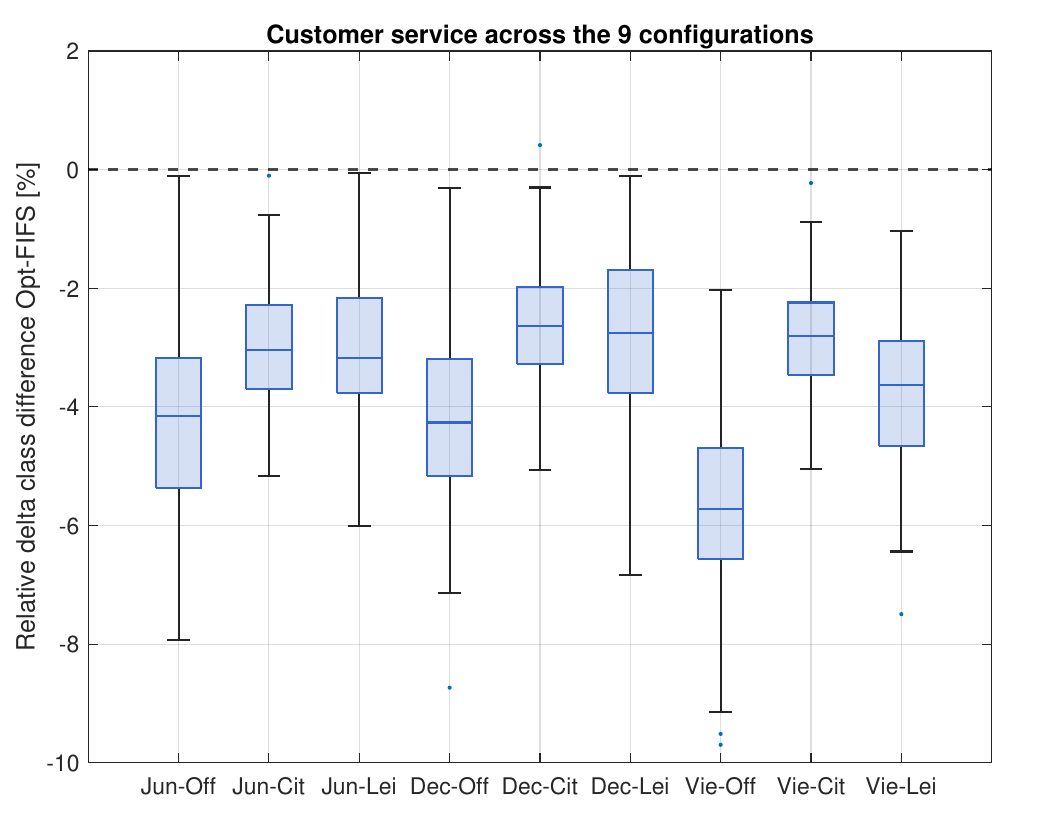} 
        \caption{100-simulations customer service across nine configurations.}
        \label{fig:box_relative_delta_class}
    \end{subfigure}
    \hfill
    \begin{subfigure}{0.48\textwidth}
        \centering
        \includegraphics[scale=0.42]{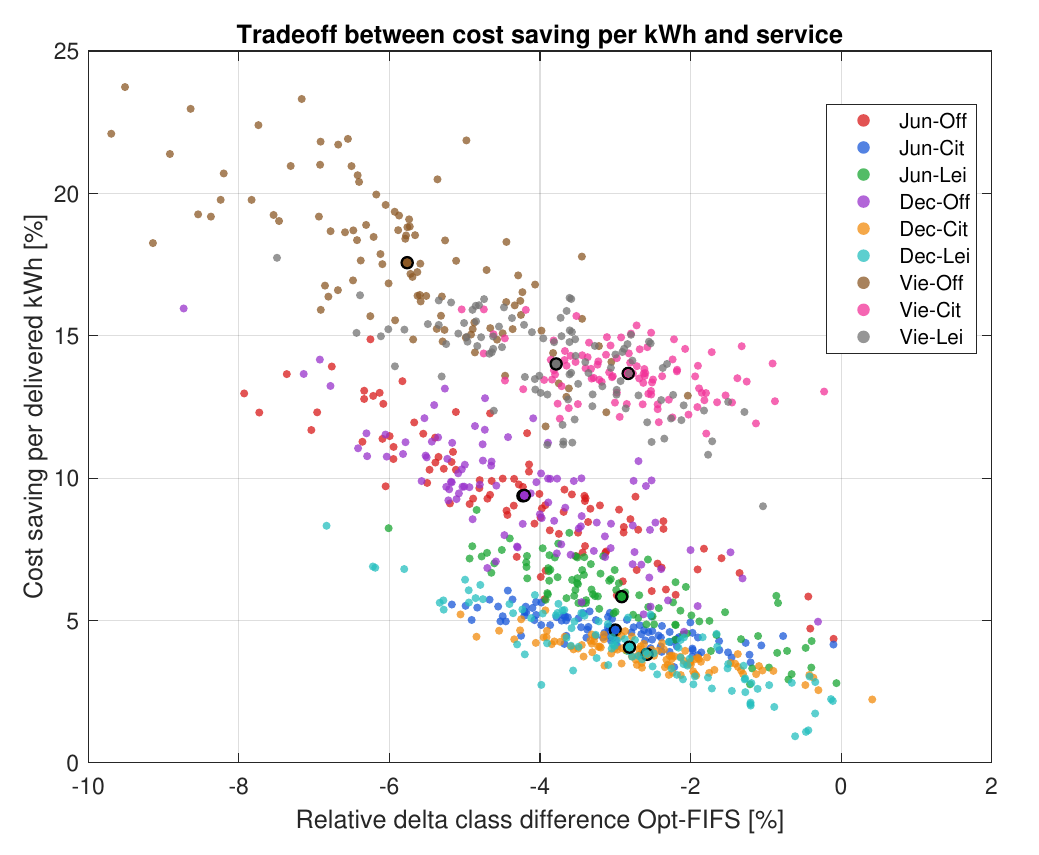} 
        \caption{Tradeoff between cost-per-kWh and customer service: 100-simulations for each configuration.}
        \label{fig:scatter_performances}
    \end{subfigure}
    \caption{Top panel: cost saving per kWh delivered across the nine configurations. Mid panel: customer service across the nine configurations. Bottom panel: tradeoff between cost saving per kWh and customer service.}
    \label{fig:100simulations}
\end{figure}

Table \ref{tab:scenarios} summarizes the average results of key evaluation metrics, for each of the nine configurations considered in our 100-scenario validation campaign, highlighting significant cost saving obtained by our proposed optimization framework compared to the FIFS approach. The reported results are obtained by averaging the performance across all scenarios: the average price-per-kWh entry in Table~\ref{tab:scenarios}, as well as the others indicators, is computed scenario-wise and then averaged over the 100 realizations. An extended version of this table, including more comparison metrics, is presented in Appendix E.

The results are consistent with the aggregate trends shown by the boxplots of Figure \ref{fig:box_savingperkWh} and \ref{fig:box_relative_delta_class}. 
Our proposed optimization approach (Opt) maintains a clear economic advantage over FIFS in every configuration, with only a modest degradation in relative delta class. Across all configurations, Opt consistently reduces both the total cost and the price per kWh with respect to FIFS, with total cost savings ranging from about 9\% to almost 30\%. As we noticed in the previous boxplots, the strongest economic improvements are obtained in the \emph{Vietnam} scenarios,  thanks to a significant difference between peak hours and off-peak hours prices. For what concerns arrival profiles, \emph{Office} appears to be the best in terms of cost saving, with its lowest value being 18.68\% in \emph{June}. This confirms that the proposed strategy successfully exploits the cost-saving potential of the shrinking-horizon optimization while preserving a largely comparable charging service.
%
%
Overall, the results support the effectiveness of the proposed approach as a practical compromise between operating cost and user satisfaction.

\begin{table*}[t!]
\caption{100 scenarios simulation key results: every number in the table is computed by averaging 100 different simulations with different realizations of the uncertain factors (i.e. EVs arrivals, departures and PV production, as seen in Section \ref{subsec:random_quant}). The "Cost saving" and "Saving per kWh" percentages are intended for the Opt model compared to FIFS.}
\label{tab:scenarios}
\centering
\scriptsize
\setlength{\tabcolsep}{2pt}
\renewcommand{\arraystretch}{0.92}
\resizebox{\textwidth}{!}{%
\begin{tabular}{l*{9}{cc}}
\toprule
& \multicolumn{2}{c}{Jun--Off}
& \multicolumn{2}{c}{Jun--Cit}
& \multicolumn{2}{c}{Jun--Lei}
& \multicolumn{2}{c}{Dec--Off}
& \multicolumn{2}{c}{Dec--Cit}
& \multicolumn{2}{c}{Dec--Lei}
& \multicolumn{2}{c}{Vie--Off}
& \multicolumn{2}{c}{Vie--Cit}
& \multicolumn{2}{c}{Vie--Lei} \\
\cmidrule(lr){2-3}
\cmidrule(lr){4-5}
\cmidrule(lr){6-7}
\cmidrule(lr){8-9}
\cmidrule(lr){10-11}
\cmidrule(lr){12-13}
\cmidrule(lr){14-15}
\cmidrule(lr){16-17}
\cmidrule(lr){18-19}
Quantity
& Opt & FIFS
& Opt & FIFS
& Opt & FIFS
& Opt & FIFS
& Opt & FIFS
& Opt & FIFS
& Opt & FIFS
& Opt & FIFS
& Opt & FIFS \\
\toprule
Saving per kWh [\%] & 9.27 & ... & 4.65 & ... & 5.81 & ... & 9.46 & ... & 3.82 & ... & 4.07 & ... & 17.47 & ... & 13.67 & ... & 13.95 & ... \\
Price per kWh [\euro/kWh] & 0.0675 & 0.0744 & 0.0963 & 0.1010 & 0.0779 & 0.0827 & 0.0670 & 0.0740 & 0.0931 & 0.0968 & 0.0826 & 0.0861 & 0.0548 & 0.0664 & 0.0796 & 0.0922 & 0.0765 & 0.0889 \\
Cost saving [\%] & 18.68 & ... & 10.67 & ... & 12.25 & ... & 19.04 & ... & 9.16 & ... & 10.36 & ... & 29.84 & ... & 18.18 & ... & 21.42 & ... \\
\midrule
Relative delta class [\%] & 52.02 & 56.27 & 53.76 & 56.76 & 53.28 & 56.20 & 51.83 & 56.04 & 54.07 & 56.65 & 53.45 & 56.26 & 50.45 & 56.21 & 53.79 & 56.62 & 52.54 & 56.32 \\
Fully charged EVs [\%] & 21.64 & 22.25 & 22.06 & 22.53 & 21.90 & 22.27 & 21.83 & 22.06 & 22.24 & 22.52 & 22.00 & 22.40 & 21.40 & 22.35 & 22.00 & 22.43 & 21.69 & 22.33 \\
\midrule
Total energy deliv. [MWh] & 13.36 & 14.90 & 29.01 & 30.97 & 16.84 & 18.08 & 13.31 & 14.91 & 29.13 & 30.86 & 16.97 & 18.17 & 12.63 & 14.86 & 29.23 & 30.85 & 16.57 & 18.14 \\
Max $P_{\text{ntw}}$ reached [MW] & 1.96 & 2.20 & 2.05 & 2.06 & 2.20 & 2.19 & 1.97 & 2.20 & 2.08 & 2.04 & 2.19 & 2.19 & 2.05 & 2.20 & 2.19 & 2.05 & 2.17 & 2.19 \\
\bottomrule
\end{tabular}
}
\end{table*}

A final interesting point is that Opt achieves cost  savings without relying on a dramatic increase in grid stress. The maximum network power in many cases is lower than the FIFS, which often reaches the $\Pmax = 2.2$ MW budget limit. Furthermore, the energy exchanged by the battery (see Table~\ref{tab:scenarios_big} in Appendix E) shows that the storage system is used by Opt much more than FIFS, which naively use PV production as a first choice. Summarizing, the proposed strategy constantly provides a well-balanced trade-off: it does not outperform FIFS in pure service metrics, but it offers a much more attractive operating point when cost and energy efficiency are taken into account.

\subsection{Comparison for a sample one-day scenario}\label{sec:one_day_scenario}
The previous section showed the potential of our proposed shrinking-horizon optimization approach by evaluating its performance over 100 stochastic simulations for each price-arrival configuration. We provide in Figure \ref{fig:single_sim_plots} a more detailed view of the system dynamics through a one-day representative sample simulation, based on the parameter set in Section~\ref{sec:Parameters_setting} and summarized in Table~\ref{tab:parameters}. We next focus on  the \emph{December/Office} configuration as a case study.

\begin{figure*}[t!]
\centering
\begin{subfigure}[t]{0.47\textwidth}
    \centering
    \includegraphics[width=\linewidth]{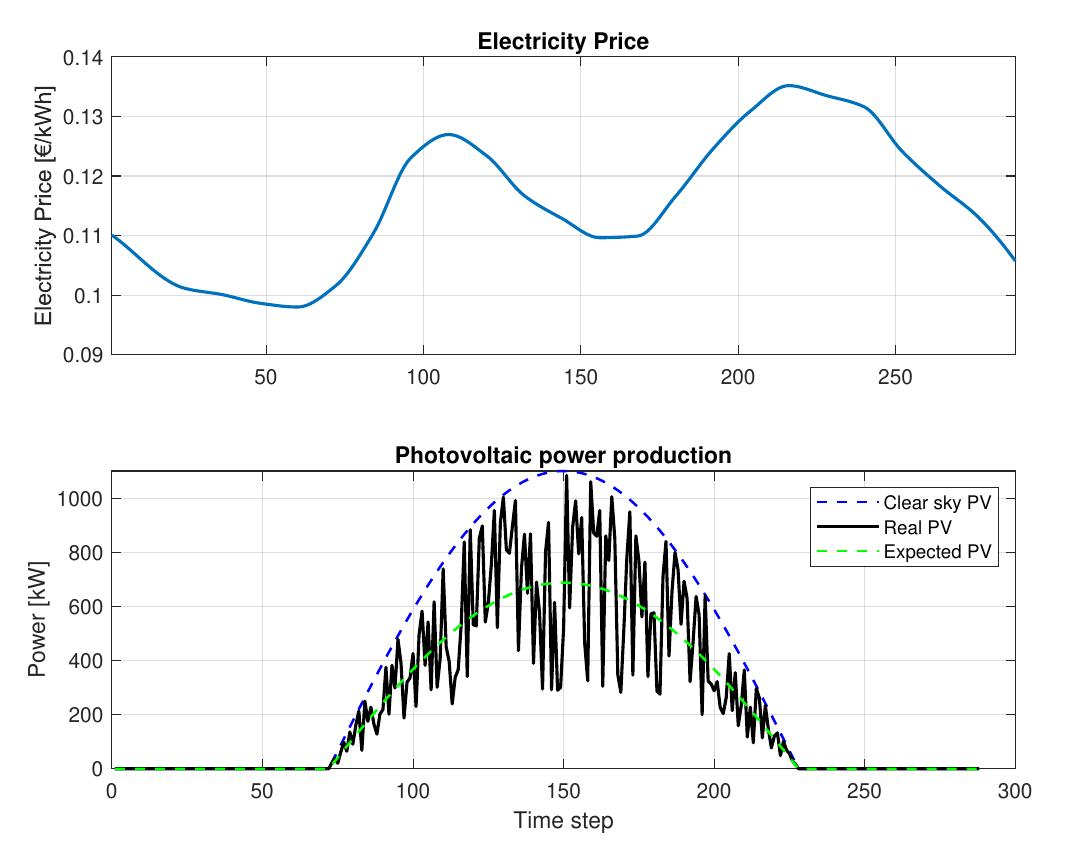}
    \caption{Electricity price (\emph{December} pattern) and PV power production.}
    \label{fig:price_PVproduction}
\end{subfigure}
\hfill
\begin{subfigure}[t]{0.47\textwidth}
    \centering
    \includegraphics[width=\linewidth]{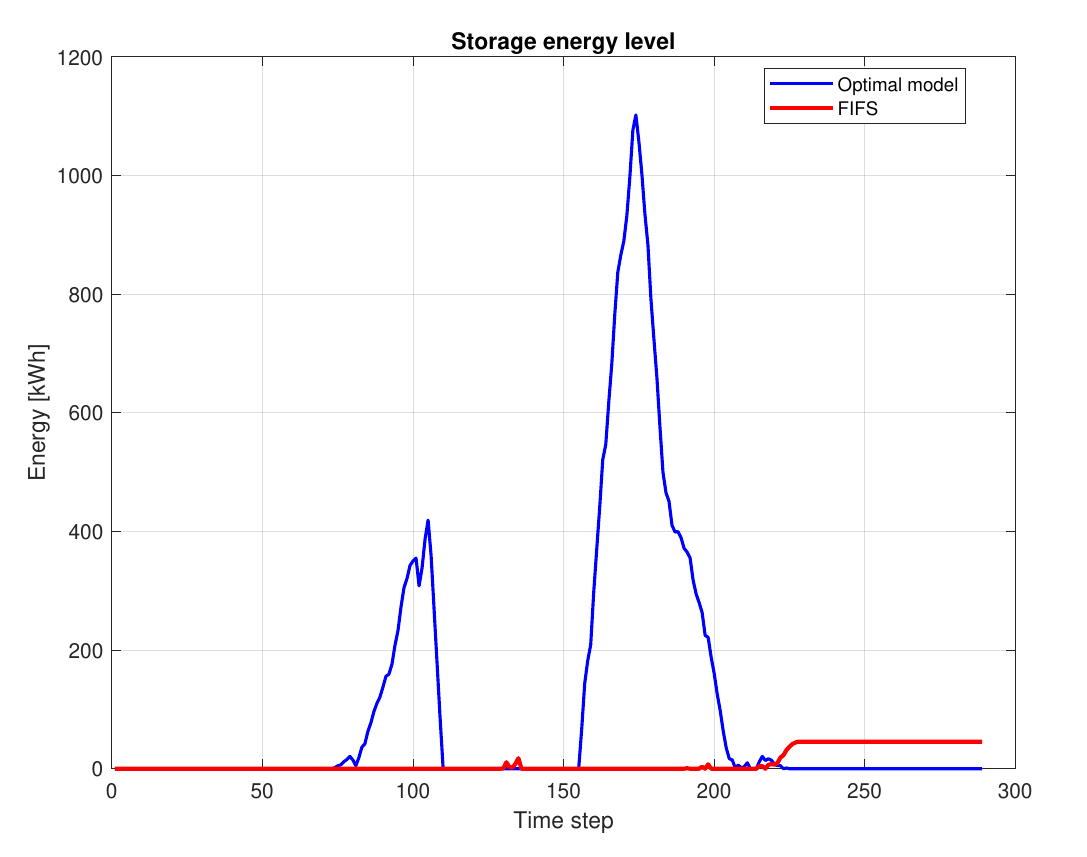}
    \caption{Battery storage energy level Opt vs FIFS.}
    \label{fig:storage_level}
\end{subfigure}

\vspace{0.3cm}

\begin{subfigure}[t]{0.47\textwidth}
    \centering
    \includegraphics[width=\linewidth]{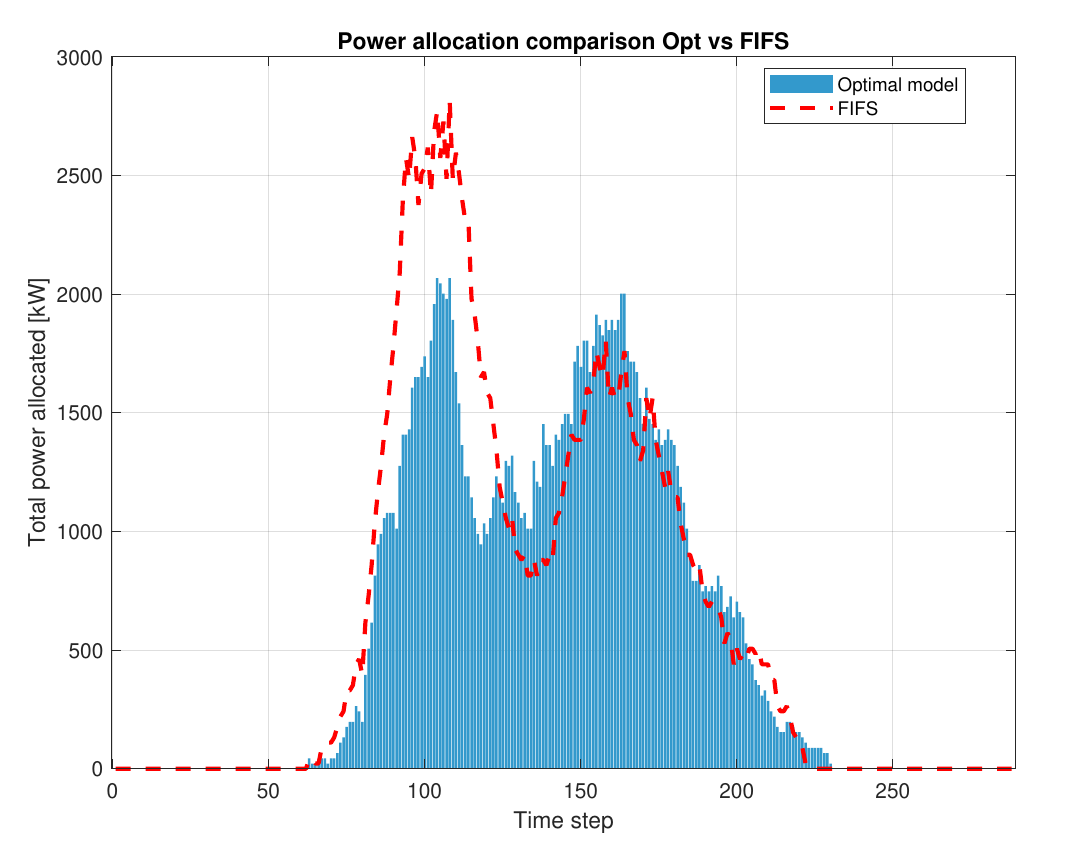}
    \caption{Total power allocation comparison Opt vs FIFS.}
    \label{fig:power_alloc_comparison}
\end{subfigure}
\hfill
\begin{subfigure}[t]{0.47\textwidth}
    \centering
    \includegraphics[width=\linewidth]{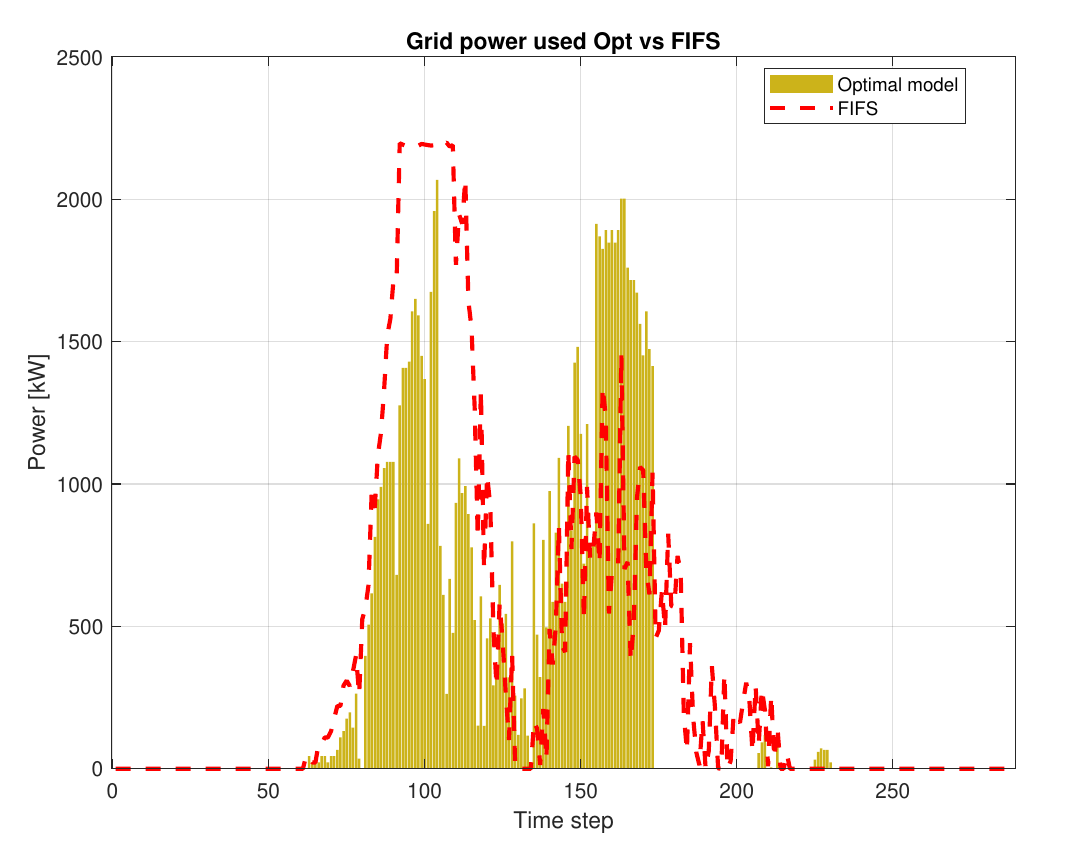}
    \caption{Power bought from the network grid Opt vs FIFS.}
    \label{fig:grid_used_comparison}
\end{subfigure}
\caption{Power system behavior of a representative one-day scenario of the \emph{December/Office} configuration.}
\label{fig:single_sim_plots}
\end{figure*}

Figure~\ref{fig:price_PVproduction} presents the averaged electricity price for \emph{December} 2025 in Italy and the PV power production over the day. The latter compares the actual PV output (black) with the clear-sky profile (dashed blue) and the average expected production (green), thereby highlighting the weather-driven variability that our forecasting layer is designed to address. Figure \ref{fig:storage_level} reports the battery energy trajectories of both FIFS and the proposed controller, showing how our approach stores PV surplus during high-generation, low-price periods and later exploits it within the same 24-hour horizon. In this way, any PV energy that is not immediately used for EV charging is temporarily stored in the battery and then reused later, reducing the need for additional grid purchases. The blue bar plot in Figure~\ref{fig:power_alloc_comparison} compares the total power allocation (grid plus clean sources) under the proposed model and FIFS. The optimal controller reduces power usage during the morning hours, when electricity prices are relatively high, and delays charging by a few hours to exploit a local price minimum. This scheduling shift also alleviates grid stress, lowering the peak network power from the 2200~kW observed under FIFS to approximately 2000~kW for a short interval in the optimal case. Later in the day, the proposed model allocates power more economically than FIFS, achieving a better overall performance, especially in this \emph{Office}-inspired occupancy scenario. Around midday, Figure \ref{fig:grid_used_comparison} shows that both strategies charge EVs mostly using PV generation, as reflected by the fact that grid energy purchases drop to nearly zero. However, the proposed controller increases grid purchases shortly after time step 150, when electricity prices become more favorable, while simultaneously beginning to store PV energy, as shown in the corresponding battery trajectory. Grid purchases then cease after time step 170, once electricity prices rise again, confirming that the proposed controller shifts part of the charging load away from expensive periods and uses PV and storage more strategically.

In this specific simulation, the proposed optimal controller delivers a substantial cost reduction, lowering the total bill from \euro1132.10 under FIFS to \euro937.27, i.e., a 17.21\% cost saving. This gain is mirrored by the reduction in the average cost per kWh, from 0.0733 to 0.0665~\euro/kWh. Importantly, this economic improvement comes at the price of only a limited reduction in service quality, with the average relative delta class dropping by just 2.60 percentage points, from 56.34\% to 53.74\%. This single-day example is representative but only illustrative; the stronger performance claims, based on the paired 100-scenario campaign, have already been settled in Section \ref{sec:validation_campaign}.

\subsection{Robustness under forecast uncertainty}
\label{subsec:robustness_simulations}

To assess the behaviour of the proposed robust LP formulation introduced in~Section~\ref{sec:uncrob}, we perform a test under forecast uncertainty. 
In particular, the nominal prediction profiles of the photovoltaic production, vehicle arrivals, and departure probabilities are perturbed within the confidence intervals as described in Section~\ref{subsec:robustness_ambiguity}.

For this evaluation example, we assume that the user has only interval-based knowledge of the expected values for PV production forecasts, EV arrivals and departures, consistently with the setting of Section~\ref{sec:uncrob}. Choosing the previously introduced \emph{December/Office} configuration as the reference scenario, these quantities are assumed at first to vary within a confidence interval of $\pm 20\%$ around such corresponding nominal values, and afterwards we enlarge the uncertainty interval up to $\pm 50\%$. In each simulation run, the expectations of the corresponding realizations are sampled as uniformly distributed random variables within their respective intervals.
We  then compare three control strategies: (a) the nominal LP from Equation~\eqref{eq:mainLPmodel}, where the user optimistically assumes to have the \emph{exact} knowledge of the aforementioned nominal expectations, corresponding to the midpoint of the confidence intervals; (b) the robust worst-case LP formulation in Equation~\eqref{eq:mainLPmodelWORSTCASE}; and (c) the FIFS baseline introduced in Section~\ref{sec:FIFS}. Both the nominal and robust controllers are implemented with $\gamma = 7$.

For each strategy and confidence interval width, we run 50 Monte Carlo realizations and collect the corresponding values of three key performance indicators: cost per delivered kWh, relative delta class, and total delivered energy. The resulting variability boxplots are reported in Figure~\ref{fig:robust_boxplots}.

\begin{figure}[t!]
    \centering

    \begin{subfigure}{0.49\textwidth}
        \centering
        \includegraphics[scale=0.43]{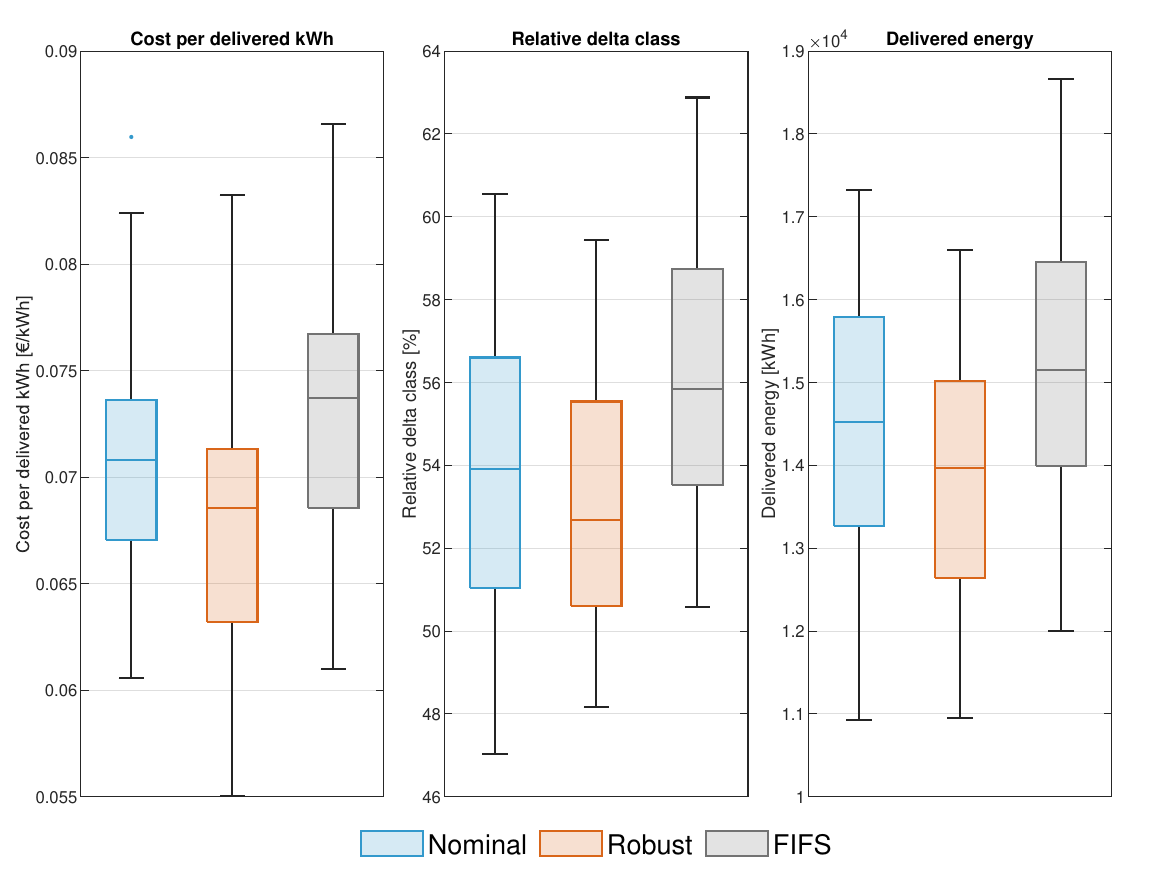}
        \caption{Results with $\pm 20\%$ confidence intervals around nominal uncertain values of EVs arrivals, departures and PV production.}
        \label{fig:robust_boxplot_20}
    \end{subfigure}

    \vspace{0.2cm}

    \begin{subfigure}{0.49\textwidth}
        \centering
        \includegraphics[scale=0.43]{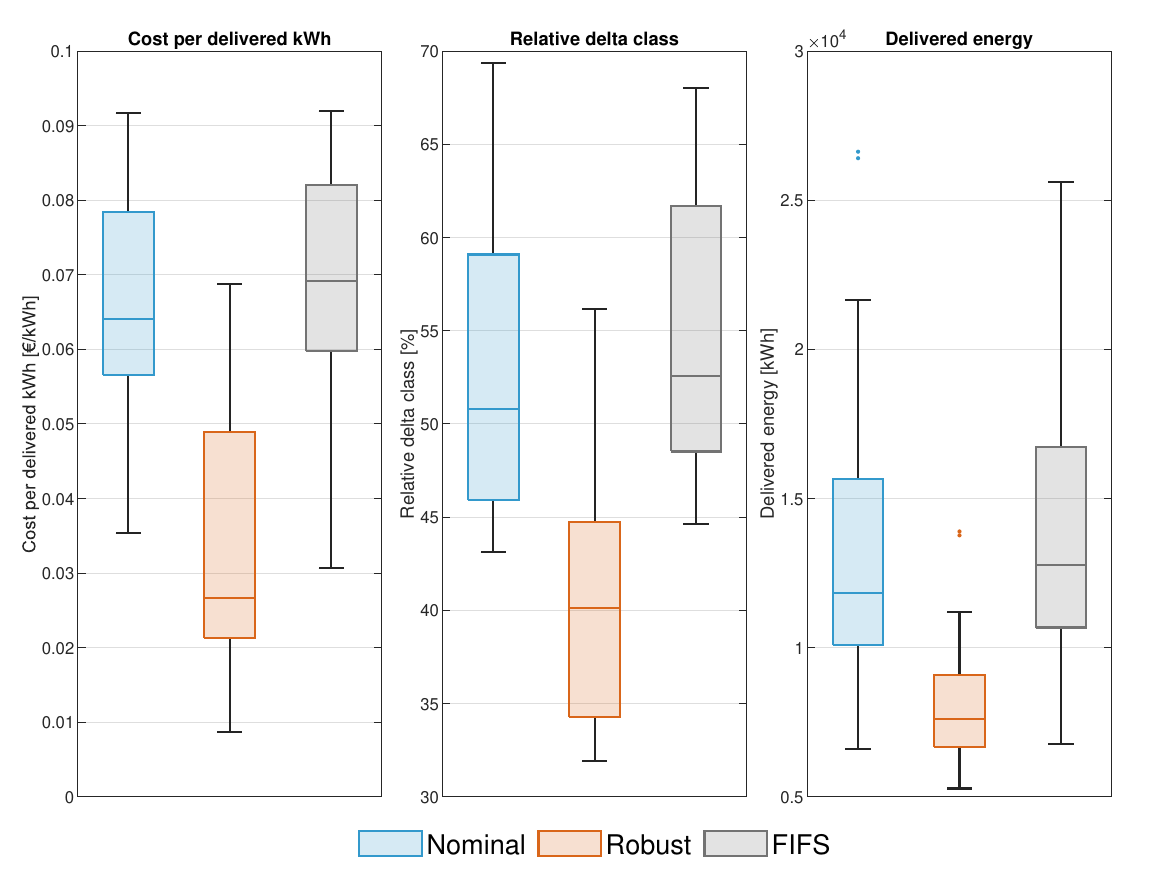}
        \caption{Results with a larger $\pm 50\%$ confidence intervals around nominal uncertain values of EVs arrivals, departures and PV production.}
        \label{fig:robust_boxplot_50}
    \end{subfigure}

    \caption{Boxplots of results obtained with 50 simulations each, with different realizations of the uncertain factors discussed in Section \ref{subsec:random_quant}.}
    \label{fig:robust_boxplots}
    
\end{figure}

The results confirm that the robust controller provides a balanced compromise between performance and conservativeness. Compared with the nominal solution, the robust approach at 20\% uncertainty level yields on average a lower cost per delivered kWh, as expected from a design tailored to a less-EVs-case scenario, along with only a marginal reduction in the average relative class delta, with the medians in Figure \ref{fig:robust_boxplot_20} differing by just one percentage point. This is consistent with the idea that a “robust user” is solving the same problem as the “nominal user”, but with less available information on future behavior of PV production, arrivals and departures. In fact, when solving the optimization problem \eqref{eq:mainLPmodelWORSTCASE}, the robust controller forecasts the future parking occupancy to be lower than it actually turns out to be, leading it to make decisions that favor purchasing less energy from the grid at the present time, which may ultimately prove suboptimal in terms of service level. 


As the uncertainty on the input factors increases, e.g., with a $\pm 50\%$ uncertainty interval on the expectations, the differences among the controllers become more pronounced. 
The boxplots in Figure~\ref{fig:robust_boxplot_50} show that  the robust controller now substantially reduces the median cost per delivered kWh, from about 0.065 \euro/kWh for the nominal controller to 0.027 \euro/kWh,
that is a 58.5\% reduction.
However, also the relative delta class decreases, although only by  about 10 percentage points with respect to the nominal controller. The overall delivered energy
also decreases, which helps alleviate stress on the power grid while maintaining a more conservative and resilient operating strategy.

These experiments highlight a key trade-off under uncertainty: the robust controller achieves a lower cost by adopting a more conservative energy purchasing strategy, while still maintaining an acceptable service level, at least at the the 20\% uncertainty level, as seen in Figure \ref{fig:robust_boxplot_20}. This behavior makes it a particularly attractive choice when economic performance must be balanced against uncertainty in EV arrivals and departures. Overall, these simulations show that the proposed robust LP is effective in preserving feasibility and stable performance under forecast errors, while keeping the reduction in nominal performance limited.

\subsection{Parametric sensitivity analysis}\label{sec:sensitivity_analysis}
In this subsection, we investigate how the system responds to variations in selected model parameters, considering some of the values  reported in Table~\ref{tab:parameters} as base parameters to be perturbed. In every plot in this section, the results corresponding to each value of the analyzed parameter are averaged over 15 different stochastic simulations. This sensitivity analysis provides further insight into the role played by the main design choices. For simplicity, in the following we focus again the configuration with Italian electricity prices from \emph{December} and \emph{Office} EVs arrival pattern.

Figure \ref{fig:sens_Pntw} shows the sensitivity of the trade-off between cost saving and customer service as the maximum network power budget $\Pmax$ changes. As expected, even a modest increase in the maximum network power budget improves the cost saving, from about 16\% at the default value $\Pmax = 2200$~kW to roughly 25--30\% when $\Pmax$ reaches 2500~kW. Conversely, when $\Pmax$ is reduced, the optimization model loses flexibility in exploiting cost-saving opportunities, but becomes increasingly closer to FIFS in terms of average relative delta class, and in some cases even outperforms it. This behavior is consistent with a smarter allocation of power across user classes under a tighter power budget, where the controller must redistribute the limited capacity more selectively. The average curve clearly shows that tighter power limits force the controller to prioritize service, whereas looser limits allow more aggressive economic cost optimization. Overall, the plot confirms that $\Pmax$ acts as a key tuning parameter governing the balance between economic benefit and user satisfaction.

    \begin{figure}[t!] 
    \centering
    \includegraphics[scale=0.49]{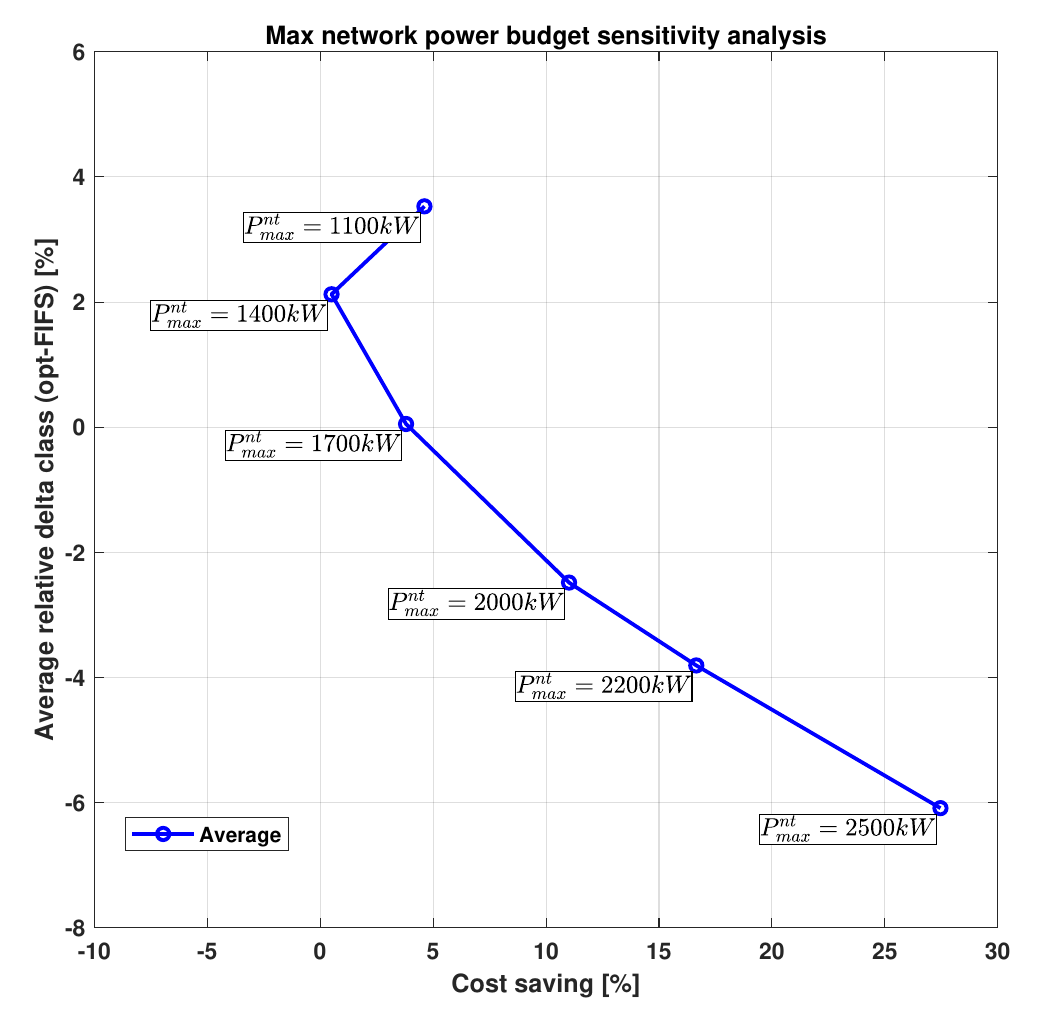} 
    \caption{Sensitivity analysis of the trade-off between cost savings and customer service as a function of the maximum network power budget $\Pmax$, averaged across 15 different stochastic scenarios.}
    \label{fig:sens_Pntw}
    \end{figure}

Following the same idea of exploiting congestion inside the charging station, Figure \ref{fig:sens_mult} reports a further analysis on the arrival intensity $\lambda_i(t)$ being scaled by a variable coefficient, while keeping all other model parameters fixed. As expected, the trade-off between economic performance and customer service is strongly affected by the number of EVs entering the system: increasing the arrival intensity tends to reduce the achievable cost saving, since the controller has less flexibility to shift charging away from expensive periods, while the average relative delta class difference compared to FIFS improves instead, proving once again that our optimal model manages to keep high standards in customer service satisfaction even when the system is congested. Conversely, reducing $\lambda_i(t)$, by applying a coefficient lower than 1, makes the charging scheduling problem less congested and allows the controller to exploit cost-saving opportunities more effectively, leading to larger economic gains up to 30-40\% compared to FIFS. Overall, the proposed strategy remains effective and reasonably robust under significant variations in the stochastic arrival process, confirming its ability to cope with uncertainty in demand volume.

    \begin{figure}[t!] 
    \centering
    \includegraphics[scale=0.49]{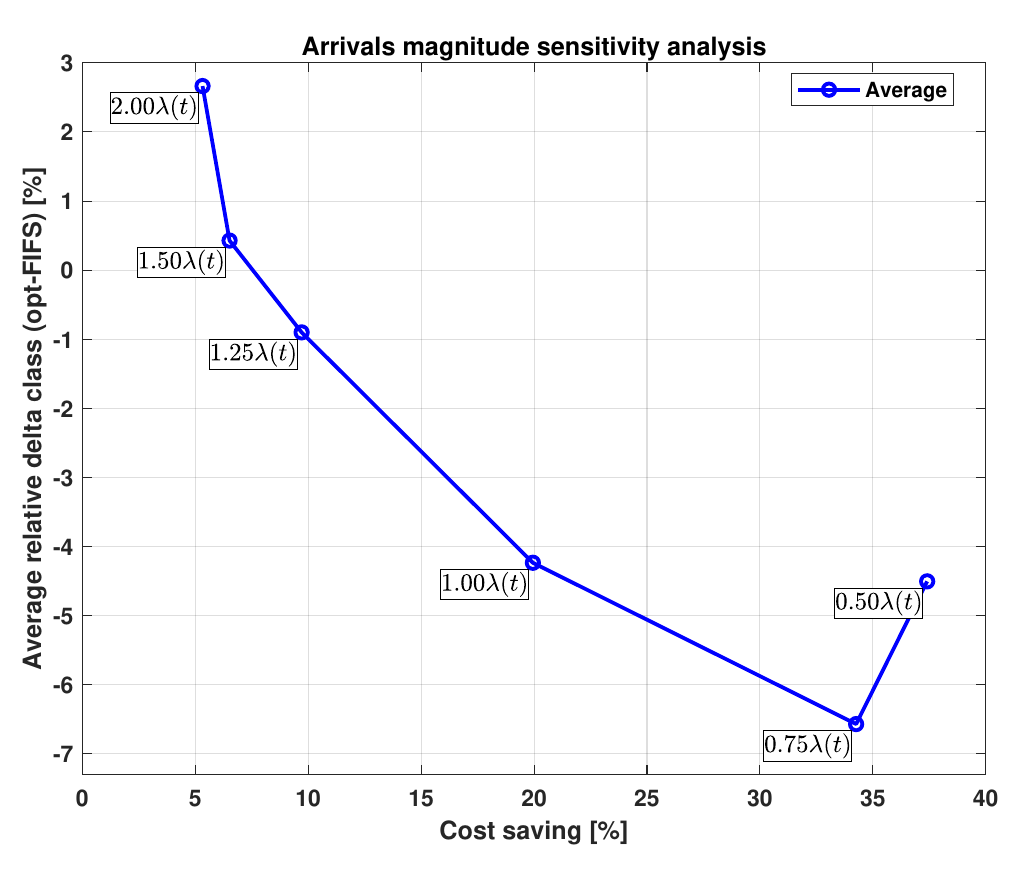} 
    \caption{Sensitivity analysis of the intensity of EVs arrival signal: $\lambda(t)$, for the configuration of \emph{Office}, is multiplied by a coefficient.}
    \label{fig:sens_mult}
    \end{figure}

These studies addressed the potential congestion at the charging station and examined how service quality and cost savings change as the parking system approaches saturation or becomes less constrained. Next,  we investigate how the energy exchanged by the system varies when parameters related to the PV production forecast or the battery storage capacity are modified. 

Figure~\ref{fig:sens_smax} shows the effect of varying the maximum storage capacity $s_{max}$. FIFS is largely insensitive to this parameter, since it always uses the available PV energy as the primary source for charging and stores it only when production exceeds the immediate charging needs; in fact, we can see that the total energy delivered by FIFS across different values of $s_{max}$ only fluctuates in an approximate 150kWh span. By contrast, the proposed controller exploits larger storage capacity more effectively: it can purchase more energy from the grid during low-price periods, store the surplus PV generation, and use it later in the day. As a result, a larger storage system enables the controller to deliver more energy to the EVs over the course of the day, achieving potentially higher cost-saving and customer satisfaction.

Figure~\ref{fig:sens_w} reports the sensitivity of the proposed shrinking horizon framework with respect to the parameter $w$, which models the uncertainty in the PV production forecast in Equation \eqref{eq:PV_production_expectation}. Smaller values of $w$ correspond to a more accurate prediction and less variability of the daily PV generation profile, and this translates into a clear improvement in the cost-saving performance up to 45\% compared to FIFS. In particular, the total energy purchased from the grid decreases monotonically as $w$ is reduced, moving from about 10 MWh when $w=1$ to roughly 3.6 MWh when $w=0$, while the cost saving increases from around 7\% to nearly 47\%. When $w$ is higher and we have bigger uncertainty in the PV production, our dynamic controller is led to buy more energy from the network grid to charge EVs and keep high service level standards, since the quantity of PV power generation for future hours is more uncertain. This confirms that our optimal controller is able to exploit more reliable PV forecasts more effectively, by planning charging actions in a way that better anticipates the available renewable energy.

    \begin{figure}[t!] 
    \centering
    \includegraphics[scale=0.48]{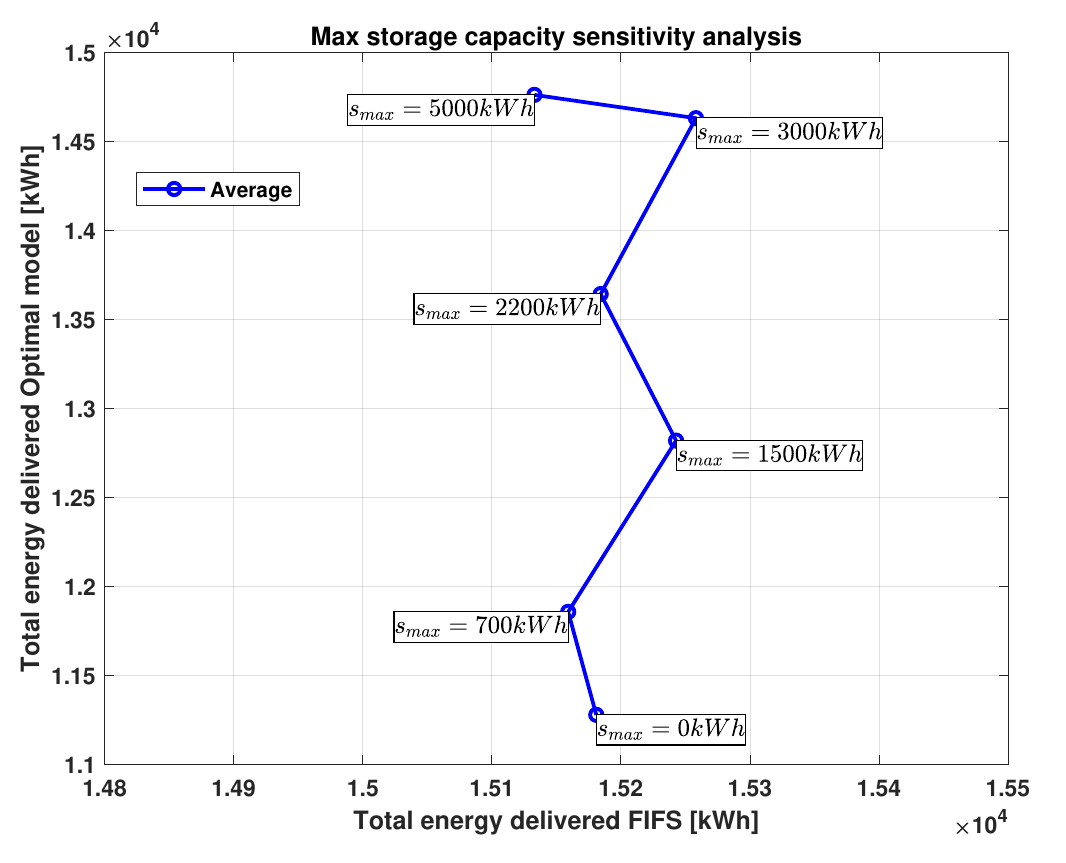} 
    \caption{Sensitivity analysis of the storage maximum capacity $s_{max}$, with values ranging from 0 to 5000 kWh; models behavior comparison Opt vs FIFS.}
    \label{fig:sens_smax}
    \end{figure}

    \begin{figure}[t!] 
    \centering
    \includegraphics[scale=0.48]{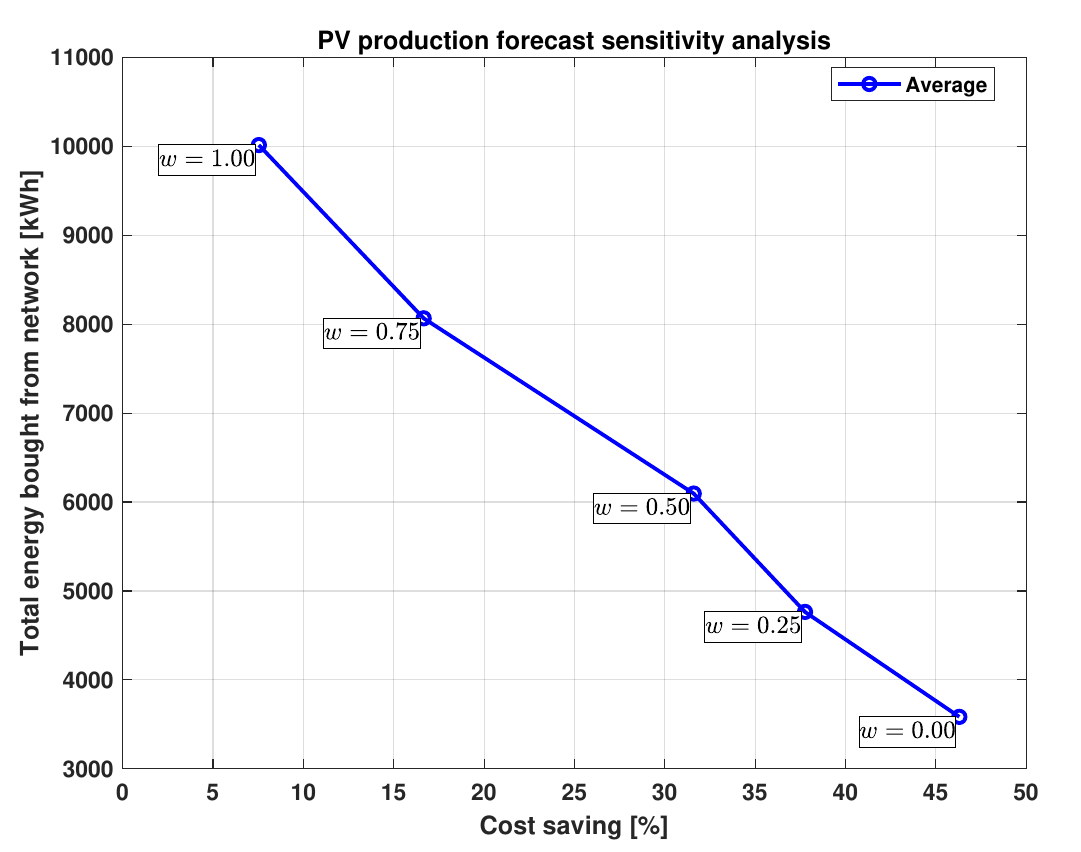} 
    \caption{Performances metrics of the proposed optimization model with respect to variations in the parameter $w$ related to the PV power production uncertainty (see Equation \eqref{eq:PV_production_expectation}).}
    \label{fig:sens_w}
    \end{figure}

Overall, the sensitivity analysis results confirm the adaptability of the proposed optimization framework under varying system conditions. The controller effectively adjusts to changes in key parameters, balancing cost efficiency and service quality. It consistently exploits additional flexibility when available while maintaining good performance under tighter constraints. This demonstrates its reliability and effectiveness in realistic, uncertain operating scenarios.


\section{Conclusions}\label{sec:Conclusions}
This paper introduced a stochastic smart-charging framework for EV charging stations equipped with PV generation and battery storage. The key modeling idea is to aggregate vehicles into classes defined by their residual charging demand, which leads to a novel compact representation that remains tractable even when the number of vehicles is large. On top of this class-based stochastic model, we formulated an expectation-based finite-horizon control problem which has the form of a linear program, and we deployed it dynamically in shrinking-horizon fashion. We also discussed a worst-case counterpart that preserves feasibility under interval ambiguity in the first moments of the underlying stochastic processes.

From a technical viewpoint, the paper shows that the aggregate arrival-charge-departure model is positive and integer-consistent, that the coupled energy-management problem can be expressed within a linear-programming framework, and that a robust implementation can be obtained without abandoning tractability.

From a practical and numerical viewpoint, the numerical study clarifies the behavior of the controller under two complementary settings: a reconstructed benchmark close to~\cite{Hermans2024}, which highlights the difference between peak-oriented and cost-oriented objectives, and a  large-fleet extensive simulation setup
oven nine realistic price/arrival configurations. The numerical evidence supports three technically sound conclusions:

\begin{itemize}[leftmargin=*]
    \item The shrinking-horizon LP is computationally well structured and can be instantiated on realistic station sizes thanks to the class-based representation; the optimization dimension depends on the number of classes chosen rather than on the number of vehicles. In the small-fleet benchmark reconstructed from~\cite{Hermans2024}, the proposed controller exhibits the expected economic behavior: compared with a peak-oriented batch MPC, it  concentrates charging demand in fewer intervals when this decreases the average price of purchased energy. Conversely, the batch MPC achieves substantially lower peaks because peak shaving is embedded directly in its objective. This confirms that the two controllers solve genuinely different control problems.
    
    \item The large-fleet validation campaign confirms the superiority of the proposed framework in balancing economic efficiency and user satisfaction. By simulating 900 stochastic scenarios across nine distinct configurations of electricity prices and arrival patterns, the analysis  demonstrates that the controller consistently outperforms the service-greedy, price-agnostic FIFS baseline. Specifically, the proposed approach achieves total cost savings ranging from 10\% to nearly 30\% and cost-per-kWh reductions of up to 17.5\%. Crucially, these substantial economic gains are attained with only a marginal reduction in pure service quality, proving that the algorithm successfully navigates the Pareto trade-off between operating costs and user satisfaction.
    
    \item Finally, the robustness example in Section~\ref{subsec:robustness_simulations} and the parametric sensitivity analysis in Section~\ref{sec:sensitivity_analysis} highlight the resilience of the proposed approach to forecast uncertainty, as well as the potential performance gains that could be achieved with more accurate predictions of the uncertain inputs. Tightened grid budgets or variations in PV forecasts naturally steer the model to prioritize safe service or aggressive cost-saving, validating the framework's operational adaptability.
\end{itemize}

Future research will focus, from the modeling point of view, on refining the proposed model to better capture real-world scenarios. From a numerical perspective, exploring alternative methods to mitigate the suboptimality induced by discretization represents a promising direction for further investigation.

\appendix

\section{}
\noindent\textbf{Proof of Lemma~\ref{lemma:DiscretizationSteo}.}
The equality constraint $P^0 \one\tran \wt c(\tau)  = \wt p\ped{ntw}(\tau) + \wt p\ped{phq}(\tau)$, follows by definition of  $\wt p\ped{phq}(\tau)$ in~\eqref{eq:PphqMigliorato}. The nonnegativity of $\wt p\ped{ntw}(\tau)$ is straightforward from its definition in~\eqref{eq:PntwMigliorato}. Moreover, it can be seen that
\begin{equation}\label{eq:Inequalityppedntw}
\wt p\ped{ntw}(\tau)\leq  p^\star\ped{ntw}(\tau).
\end{equation}
If $\wt p\ped{ntw}(\tau)=0$ this is trivially true, otherwise we have  $\wt p\ped{ntw}(\tau)=P^0 \one\tran \wt c(\tau)  -  p^\star\ped{phq}(\tau)\leq P^0 \one\tran  c^\star(\tau)  -  p^\star\ped{phq}(\tau)= p^\star\ped{ntw}(\tau)$.
The inequality $\wt p\ped{ntw}(\tau) \leq  \Pmax$ is then satisfied since $p^\star\ped{ntw}(\tau) \leq  \Pmax$.
Moreover, we also have
\begin{equation}\label{eq:InequalityppednPHQ}
\wt p\ped{phq}(\tau)\leq  p^\star\ped{phq}(\tau).
\end{equation}
since by definition
\[
\begin{aligned}
\wt p\ped{phq}(\tau)&= P^0 \one\tran \wt c(\tau)  -  \wt p\ped{ntw}(\tau)=\min\{p^\star\ped{phq}(\tau) ,P^0 \one\tran \wt c(\tau)\}\\&\leq  p^\star\ped{phq}(\tau).
\end{aligned}
\]
Let us now reason by cases on the charge variable $\wt q(\tau)$. \\ \emph{(If $q^\star(\tau)\leq  \wt p\ped{phq}(\tau) - \bar P\ped{ph}(\tau)$):} In this case it is easy to see that $\wt q(\tau)=q^\star(\tau)$ since $q^\star(\tau)=\frac{\bar{s}(t)-\bar{s}(t+1)}{\Delta}\geq \frac{\bar{s}(t)-s\ped{max}}{\Delta}$ and $q^\star(\tau)\geq -q\ped{max}$. Then, in this case, the inequalities 
\[
0\leq \wt p\ped{phq}(\tau)\leq \bar P\ped{ph}(\tau) + \wt q(\tau)\quad \text{ and }\quad |\wt q(\tau)| \leq q\ped{max}, 
\]
are satisfied, since by~\eqref{eq:InequalityppednPHQ} we have
\[
0\leq \wt p\ped{phq}(\tau)\leq  p^\star\ped{phq}(\tau)\leq \bar P\ped{ph}(\tau) + q^\star(\tau)= \bar P\ped{ph}(\tau) + \wt q(\tau)
\]
and $|q^\star(\tau)|\leq q\ped{max}$.\\
\emph{(If $q^\star(\tau)>  \wt p\ped{phq}(\tau) - \bar P\ped{ph}(\tau)$):} In this case, by definition we have 
\[
\wt q(\tau) = \max \Big\{  \wt p\ped{phq}(\tau) - \bar P\ped{ph}(\tau), -q\ped{max},\; \frac{\bar s(\tau) - s\ped{max}}{\Delta}\Big\}
\]
and thus 
\[
\wt q(\tau) \geq  \wt p\ped{phq}(\tau) - \bar P\ped{ph}(\tau)\;\;\Rightarrow \;\;\wt p\ped{phq}(\tau) \leq  \bar P\ped{ph}(\tau)+\wt q(\tau).
\]
Since $|q^\star(\tau)|\leq q\ped{max}$ and $q^\star(\tau)=\frac{\bar{s}(t)-\bar{s}(t+1)}{\Delta}\geq \frac{\bar{s}(t)-s\ped{max}}{\Delta}$ we also have that $|\wt q(\tau)|\leq q\ped{max}$, as required.\\
The inequality 
\[
0\leq \wt c(\tau)\leq x(\tau)
\]
follows by the fact that $0\leq  c^\star(\tau)\leq x(\tau)$ and $\wt c(\tau)=\lfloor c^\star(\tau) \rfloor\in \N^{n+1}$. 
The second part of the statement straightforwardly follows by Lemma~\ref{lemma:Preliminary} and by the fact that $\frac{\bar{s}(t)-s\ped{max}}{\Delta}\leq \wt q(\tau)\leq q^\star(\tau)$.
\qed
\section{}

\noindent\textbf{Proof of Proposition~\ref{prop_ineqversion}.}
Let $\mathcal J(\cdot)$ denote the objective function of problems in~\eqref{eq:mainLPmodel} and~\eqref{eq:mainLPmodel_ineq}. Since the feasible set of~\eqref{eq:mainLPmodel} is contained in the feasible set of~\eqref{eq:mainLPmodel_ineq}, one always has
\[
p^\star_{rel,\tau}\le p^\star_\tau.
\]
We prove the two statements in turn.

\emph{Proof of Item~2.}
Assume $\gamma>0$ and $\beta_i>0$ for all $i=0,\dots, n$. Let
$C^\star,Q^\star,P^\star_{\mathrm{ntw}},P^\star_{\mathrm{phq}},X^\star,S^\star$
be an optimal solution of~\eqref{eq:mainLPmodel_ineq}. We show that the state inequalities must hold with equality.

Assume by contradiction that they do not. Let $\tau_1\in\{\tau,\ldots,T\}$ be the earliest time at which a strict inequality appears, namely either
$\bar x^\star(\tau_1)\lneqq x(\tau)$ if $\tau_1=\tau$,
or
\[
\bar x^\star(\tau_1)\lneqq \bar A(\tau_1-1)\bigl[\bar x^\star(\tau_1-1)+Bc^\star(\tau_1-1)\bigr]+\bar a(\tau_1-1)
\;
\]
if $\tau_1>\tau$.
Define $w\doteq x(\tau)$ if $\tau_1=\tau$, or 
\[
w\doteq \bar A(\tau_1-1)\bigl[\bar x^\star(\tau_1-1)+Bc^\star(\tau_1-1)\bigr]+\bar a(\tau_1-1),  
\]
if $\tau_1>\tau$.
Choose the smallest index $j\in\{0,\ldots,n\}$ such that $\bar x_j^\star(\tau_1)<w_j$. Now define a modified state  by
\[
\tilde x_i(\tau_1)=\begin{cases}
\bar x_i^\star(\tau_1), & i\neq j,\\
w_j, & i=j,
\end{cases}
\]
At all other times $\tau< \tau_1$ (if any) we keep the same state and input variables.

By construction, all constraints at time $\tau_1$ remain feasible, since 
$c^\star(\tau_1)\le x^\star(\tau_1) < \tilde{x}(\tau_1)$. 
If $\tau_1<T$, the forward state is then modified inductively by defining
\[
\wt x(t+1)=  \bar A(t)\bigl[\tilde x(t)+Bc^\star(t)\bigr]+\bar a(t),\;\;\forall t\in \{\tau_1+1,\dots, T\}. 
\]
By monotonicity of the average system proven in~Lemma~\ref{lemma:Preliminary} (recall equation~\eqref{eq:monoticity}), we have that
\[
\tilde x(t)\geq \bar x^\star(t)\geq c^\star(t), \;\;\forall t\in \{\tau_1+1,\dots, T\}.
\]

Hence we have constructed another feasible solution of~\eqref{eq:mainLPmodel_ineq} with the same variables as the optimum except the state variable that satisfy
$\tilde x(\tau_1)\gneqq\bar x^\star(\tau_1)$ and $\tilde x(t)\geq\bar x^\star(t)$, for all $t\in \{\tau_1+1,\dots T\}$.  Since $\gamma>0$ and $\beta>0$, this implies
\[
\begin{aligned}
\mathcal J(C^\star,\tilde X, \,Q^\star,S^\star,P^\star_{\mathrm{ntw}},P^\star_{\mathrm{phq}})&= \hspace{-0.1cm}\sum_{t=\tau}^{T-1}\bar \pi_{t} \Delta \bar p^\star\ped{ntw}(t)\hspace{-0.05cm} - \hspace{-0.05cm}\gamma \sum_{t=\tau}^{T}\beta \tilde x(t)< \sum_{t=\tau}^{T-1}\bar \pi_{t} \Delta \bar p^\star\ped{ntw}(t) - \gamma \sum_{t=\tau}^{T}\beta \bar x^\star(t)\\
&=
\mathcal J(C^\star,X^\star,Q^\star,S^\star,P^\star_{\mathrm{ntw}},P^\star_{\mathrm{phq}}),
\end{aligned}
\]
contradicting optimality. Therefore every optimal solution of~\eqref{eq:mainLPmodel_ineq} must satisfy the $x$-dynamics with equality, which proves Item~2.\\
\emph{Proof of Item~1.}
Assume now only $\gamma\ge 0$ and $\beta_i\ge 0$. Start from an arbitrary optimal solution of~\eqref{eq:mainLPmodel_ineq}. Repeating the slack-removal construction used above whenever a strict inequality is present in the $x$-dynamics, we obtain another optimal solution of~\eqref{eq:mainLPmodel_ineq} that satisfies
\[
\bar x(t+1)=\bar A(t)\bar x(t)+\bar a(t)+\bar A(t)Bc(t),\qquad \bar x(\tau)=x(\tau).
\]
The cost function cannot increases during this procedure, as proven above.
Therefore we have constructed an optimal solution of~\eqref{eq:mainLPmodel_ineq} that is also feasible for~\eqref{eq:mainLPmodel}. This yields $p^\star_\tau\le p^\star_{rel,\tau}$.
Combined with the opposite inequality proved at the beginning, we conclude that
$p^\star_{rel,\tau}=p^\star_\tau$,
which proves Item~1.
\qed
\section{}
In this appendix, we present the parameter choices defining the additional functions and stochastic processes involved in the simulations performed in Section~\ref{sec:NumericalSimulation}.
\begin{itemize}[leftmargin=*]
    \item \textbf{Instantaneous maximum power $P^0$:} using 22 kW three- phase alternating current (AC) charge points is appropriate for European public and workplace settings because the majority of publicly accessible chargers in the European Union (EU) are still AC and most standard AC posts are installed up to 22 kW with Type 2 connectors mandated across Europe. Recent EU/EAFO reporting \cite{EAFO_RechargingSystems_2021} shows that only about one in eight public chargers is “fast” direct current (DC), implying that the vast remainder are AC, typically up to 22 kW; this makes 22 kW a representative and policy‑aligned assumption for modeling public/semipublic charging \cite{ACEA_OneInEightFast_2024}.
    
    \item \textbf{Entry class probability $p_i$:} the probability of vehicles arriving in each class, denoted as $p_i$ (see \eqref{eq:arrival_expectation}), is assumed to be time-invariant under the assumption that no EV arrives fully charged, i.e., $p_0 = 0$. The exact values are computed based on a real-world dataset from Shenzhen, China \cite{Shenzen_data}. This dataset was collected from 30 real public EV charging stations covering thousands of EVs; we extrapolated the EVs State of Charge (SoC) from the dataset, divided into $n=30$ sets as shown in Figure \ref{fig:class_proportions}.

    \begin{figure}[h!] 
    \centering
    \includegraphics[width=0.49\textwidth]{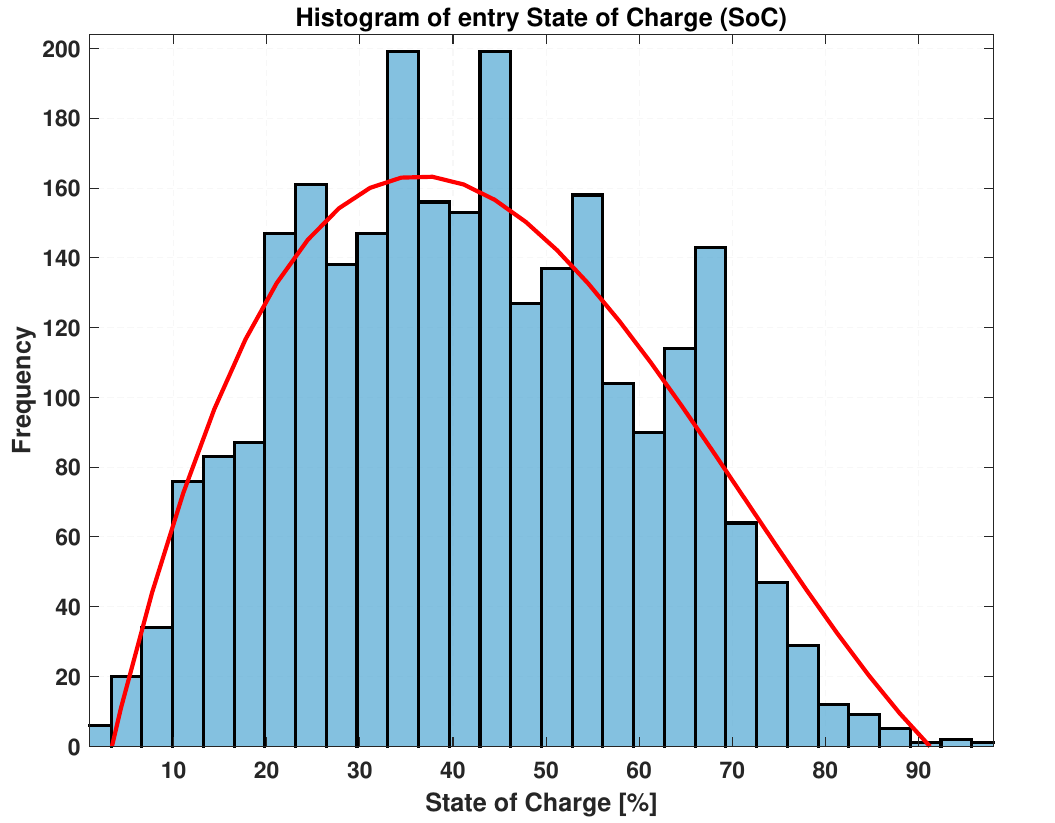} 
    \caption{SoC values obtained from the dataset in \cite{Shenzen_data}.}
    \label{fig:class_proportions}
    \end{figure}

  The class probabilities for our model were obtained at first by fitting the actual collected SoC, as shown by the red curve in Figure \ref{fig:class_proportions}, and then flipped and normalized such that $\sum_i p_i = 1$. The flipping is done so that higher classes correspond to lower SoC values, since the idea is that empty EVs (with low SoC) require more charging cycles and therefore should be assigned to higher-value classes.

    \item \textbf{Departure rates $\alpha_i(t)$:} these rates express the probability of vehicles in class $i$ departing at time $t$. They can be estimated from empirical dwell time distributions. Histogram-based analysis of past charging sessions can be used to derive time-varying departure rates.\\
    At first, we suppose to model vehicle departures in discrete time with a constant per-step departure probability \(\alpha\). At each time step of length \(\Delta\) (in hours), a vehicle independently leaves the station with probability \(\alpha\). Let \(K\) denote the number of time steps that a vehicle remains in the station before departure. Under this assumption, \(K\) is a geometrically distributed random variable with parameter \(\alpha\), and its expectation is $\mathbb{E}[K] = \frac{1}{\alpha}$.
    The corresponding expected parking time is \[ \mathbb{E}[T] = \mathbb{E}[K]\times \Delta = \frac{\Delta}{\alpha}. \] If we want the average parking duration to be equal to a desired value \(k\) (in hours), we can invert the above relation and obtain the constant departure probability 
    \begin{equation} 
    \alpha = \frac{\Delta}{k}. 
    \label{eq:alpha_from_k} 
    \end{equation} 
    For example, with a time-step length \(\Delta = 5\) minutes (i.e., \(\Delta = 1/12\) hours) and a target average parking time \(k = 3\) hours, we get \(\alpha = (1/12)/3 \approx 0.028\), meaning that at each time step a vehicle has a probability of about \(2.8\%\) of leaving the station.

    In practice, departure behavior is class-dependent, in the sense that EVs that are almost fully charged, or fully charged, are more likely to depart sooner than others. Therefore, following the above reasoning about $\alpha$, we set a value for the departures parameter $\alpha_i (t)$ such that $\alpha_n (t) \approx0.03, \forall t$, and the values grow when moving toward lower classes. Specifically, we used $\alpha_i(t) = 0.15 \times \alpha_i^{\mathrm{base}}, \forall t$, where  
    \begin{equation}
    \alpha_i^{\mathrm{base}} = \frac{n+1-i}{n+2}, \quad i = 0,1,\dots,n.
    \label{eq:alpha_base_exp}
    \end{equation}
However, it is worth noting that the model can also accommodate other types of dependence in the departure probabilities, if desired.

    \item \textbf{Storage maximum $s_{\mathrm{max}}$:} this parameter represents the maximum energy capacity of the stationary battery system. The stationary battery is generally sized to provide one or a few hours of peak shaving for the aggregated EV charging demand, consistently with existing sizing studies of PV--storage-supported EV charging stations \cite{Liu2020_OptimalSizingPVStorageEV}. This choice allows the charging infrastructure to buffer short-term fluctuations in load and PV generation, while keeping investment costs within a realistic range for commercial applications. In addition, the storage capacity is sufficient to shift a significant share of the daily energy consumption away from high-price periods, which is in line with typical design guidelines for behind-the-meter battery systems. For all these reasons, we choose $s_{\mathrm{max}} = 2.2$ MWh, which corresponds to  one hour of full-power station operation at $\Pmax = 2.2$ MW.

    \item \textbf{Photovoltaic pattern $P_{ph}(t)$ and $q_{\mathrm{max}}$:} these represent the PV pattern and the maximum power exchangeable with the storage system at any time step. The PV generation profile over 24 hours is modeled as a normalized clear-sky curve with a single midday peak and zero generation at night. In the large-fleet case study we choose a peak PV power equal to $\Pmax/2 = 1.1$ MW, i.e., a sizable but still clearly limited on-site resource compared with the contracted grid power. The clear-sky synthetic photovoltaic power profile $\hat P_{ph}(t)$ is modeled using a concave sinusoidal function as
    \begin{equation}
    \hat P_{ph}(t) =
    \begin{cases}
    0,
    & \rule{-.4cm}{0cm}{t\Delta < t_{\mathrm{start}} \ \text{or}\ t\Delta > t_{\mathrm{end}}}, \\[4pt]
    {\frac{\Pmax}{2}}
    \sin\!\Bigl(
      \dfrac{\pi\,({t\Delta - t_{\mathrm{start}}})}{t_{\mathrm{end}} - t_{\mathrm{start}}}
    \Bigr),
    & {t_{\mathrm{start}} \le t\Delta \le t_{\mathrm{end}}},
    \end{cases}
    \label{eq:pv_hat}
    \end{equation}
    where $t_{\mathrm{start}} = 6$ and $t_{\mathrm{end}} = 19$ denote the beginning and end
    of the PV generation window (in hours). A stochastic realization of the PV power $P_{ph}(t)$ is then obtained by
    subtracting a random uniform-distributed attenuation factor $\xi_t \sim U(0,w \hat P_{ph}(t))$ (see \ref{subsec:random_quant}), where we chose $w = 0.75$; using this configuration, we have that $P_{ph}(t) = \hat P_{ph}(t) - \xi_t$. This sinusoidal irradiance profile captures the main diurnal pattern of solar production while remaining simple enough to be integrated into an optimization framework with discrete time steps; this is in line with standard PV performance models~\cite{Leone2023_PVBatterySizingUFCS}. For a recent comprehensive review of methods for modeling clear-sky irradiation curves, the reader is referred to~\cite{Antonaz-Torres19}.

    By combining this PV profile with the assumed battery size, the model can realistically represent the interaction between local renewable generation, storage operation, and EV charging demand over a typical day. The maximum exchangeable power with the storage system is set to $q_{\mathrm{max}} = 1.1$ MW, corresponding to a 0.5C charge/discharge rate for a $2.2$ MWh battery.
    
    \item \textbf{Initial state $x(0)$ and $s(0)$:} the number of EVs (for each class) inside the charging station and the battery charge at the beginning of the simulation horizon. For simplicity, and since we are considering a 24 hours time span, we set $x_i(0) = 0, \forall i$ and $s(0) = 0$ kWh.

    \item \textbf{Customer satisfaction weights $\beta_i$:} these parameters reflect the priority given to different vehicle classes in the objective function. They can be chosen according to policy goals (e.g., prioritize fast charging, or prioritize high-demand customer). In Section \ref{Sec:SmartChargingControlProblem}, we considered 
    non-increasing $\beta_i$,  such that $\beta_0 = 1, \beta_i = b^i ,\forall 1 \leq i \leq n$ with $0 < b < 1$. The upper limit for $b \rightarrow 1$ is the constant vector $\beta = [1,\dots,1]$ which basically would give no information about any class preferences, while the lower limit for $b \rightarrow 0$ is the vector $\beta = [1,0,\dots,0]$, which pushes EVs towards class 0, but gives again no hint about class preferences, since all classes (except class 0) would be penalized in the same way. In our simulations we settled for the intermediate choice $b=0.5$.

    \item \textbf{Tradeoff parameter $\gamma$:} In Section \ref{sec:Pareto analysis}, we described the procedure used to select the optimal $\gamma$ values for each price-arrival pair. Table~\ref{tab:optimal_gamma} reports the optimal $\gamma$ values adopted in the campaign validation for each configuration, as illustrated in Figure \ref{fig:frontiere}.

\end{itemize}

\begin{table}[t!]
  \caption{Tradeoff parameter $\gamma$ shown in Figure \ref{fig:frontiere}.}
    \centering
    \begin{tabular}{|l|c|c|c|}
    \hline
    \textbf{Optimal $\gamma$} & {Office} & {City} & {Leisure} \\
    \hline
{June} & 4.4 & 4.6 & 4.3 \\
    \hline
    {December} & 4.3 & 4.5 & 4.5 \\
    \hline
  {Vietnam} & 3.6 & 4.3 & 4.4 \\
    \hline
    \end{tabular}
    \label{tab:optimal_gamma}
\end{table}

\section{}
The FIFS algorithm is implemented as follows.

\vspace{1em}
\hrule
\vspace{0.5em}
\captionof{algorithm}{First-In-First-Served FIFS charging with PV-first energy management}\label{alg:FIFS}
\vspace{0.5em}
\hrule
\vspace{0.5em}
\begin{algorithmic}[1]
\State \textbf{Initialization:} set the storage state $s(0)$ and initialize the global FIFO queue $\mathcal Q(0)$ of connected EVs.
\For{$t=0,\ldots,T-1$}
    \State Let $\mathcal U(t)=\{v\in\mathcal Q(t):k(v)\ge 1\}$ be the unfinished EVs, ordered by arrival time
    \State $m_t \gets \min\!\left\{|\mathcal U(t)|,\,
    \left\lfloor\frac{\Pmax + P_{\mathrm{ph}}(t)+\min\!\left\{q_{\max},\,\frac{s(t)}{\Delta}\right\}}{P^0}\right\rfloor\right\}$
    \State Select $\mathcal S(t)$ as the first $m_t$ EVs in $\mathcal U(t)$ (oldest arrivals first)
    \State Set
    $\displaystyle
    c_i(t)\gets \bigl|\{v\in\mathcal S(t):k(v)=i\}\bigr|$, $i=1,\dots,n$,
    and $c_0(t)\gets 0$
    \State Compute the charging demand $\displaystyle P_c(t)\gets P^0\sum_{i=1}^{n} c_i(t)$
    \Statex \textit{PV-first power allocation}
    \State $p_{\mathrm{ph}}(t)\gets \min\{P_{\mathrm{ph}}(t),\,P_c(t)\}$
    \State $R(t)\gets P_c(t)-P_{\mathrm{ph}}(t)$
    \If{$R(t)\ge 0$}
        \State $q(t)\gets \min\!\left\{R(t),\,q_{\max},\,\frac{s(t)}{\Delta}\right\}$
        \State $p_{\mathrm{ntw}}(t)\gets R(t)-q(t)$
    \Else
        \State $q(t)\gets -\min\!\left\{-R(t),\,q_{\max},\,\frac{s_{\max}-s(t)}{\Delta}\right\}$
        \State $p_{\mathrm{ntw}}(t)\gets 0$
    \EndIf
    \State Update the storage state $s(t+1)\gets s(t)-q(t)\Delta$
    \Statex \textit{Class update after charging}
    \For{each $v\in\mathcal S(t)$}
        \State $k(v)\gets k(v)-1$
    \EndFor

    \Statex \textit{Departures after the class transition}
    \State Let $\mathcal Q^{+}(t)$  be the queue after the class update
    \For{each vehicle $v\in\mathcal Q^{+}(t)$}
        \State Remove $v$ from $\mathcal Q^{+}(t)$ with probability $\bar\alpha_{k(v)}(t)$
    \EndFor

    \Statex \textit{Arrivals become available only at the next step}
    \State Let $\mathcal A(t)$ be the list of arrivals observed during $\mathcal I_t$, ordered by arrival time
    \State $\mathcal Q(t+1)\gets \mathrm{concat}\!\bigl(\mathcal Q^{+}(t),\,\mathcal A(t)\bigr)$
\EndFor
\end{algorithmic}
\vspace{0.5em}
\hrule
\vspace{1em}

\section{}

\begin{table*}[t!]
\caption{Campaign validation results (extension of Table \ref{tab:scenarios}). The "Cost saving" and "Saving per kWh" percentage are intended for the Opt model compared to FIFS total cost, while the average number of EVs per day is equal for both models in the same configuration.}
\label{tab:scenarios_big}
\centering
\scriptsize
\setlength{\tabcolsep}{2pt}
\renewcommand{\arraystretch}{0.92}
\resizebox{\textwidth}{!}{%
\begin{tabular}{l*{9}{cc}}
\toprule
& \multicolumn{2}{c}{Jun--Off}
& \multicolumn{2}{c}{Jun--Cit}
& \multicolumn{2}{c}{Jun--Lei}
& \multicolumn{2}{c}{Dec--Off}
& \multicolumn{2}{c}{Dec--Cit}
& \multicolumn{2}{c}{Dec--Lei}
& \multicolumn{2}{c}{Vie--Off}
& \multicolumn{2}{c}{Vie--Cit}
& \multicolumn{2}{c}{Vie--Lei} \\
\cmidrule(lr){2-3}
\cmidrule(lr){4-5}
\cmidrule(lr){6-7}
\cmidrule(lr){8-9}
\cmidrule(lr){10-11}
\cmidrule(lr){12-13}
\cmidrule(lr){14-15}
\cmidrule(lr){16-17}
\cmidrule(lr){18-19}
Quantity
& Opt & FIFS
& Opt & FIFS
& Opt & FIFS
& Opt & FIFS
& Opt & FIFS
& Opt & FIFS
& Opt & FIFS
& Opt & FIFS
& Opt & FIFS \\
\toprule
Cost saving [\%] & 18.68 & ... & 10.67 & ... & 12.25 & ... & 19.04 & ... & 9.16 & ... & 10.36 & ... & 29.84 & ... & 18.18 & ... & 21.42 & ... \\
Saving per kWh [\%] & 9.27 & ... & 4.65 & ... & 5.81 & ... & 9.46 & ... & 3.82 & ... & 4.07 & ... & 17.47 & ... & 13.67 & ... & 13.95 & ... \\
Price per kWh [\euro/kWh] & 0.0675 & 0.0744 & 0.0963 & 0.1010 & 0.0779 & 0.0827 & 0.0670 & 0.0740 & 0.0931 & 0.0968 & 0.0826 & 0.0861 & 0.0548 & 0.0664 & 0.0796 & 0.0922 & 0.0765 & 0.0889 \\
Total Cost [\euro] & 902.01 & 1109.48 & 2795.42 & 3129.86 & 1311.67 & 1495.59 & 893.15 & 1103.93 & 2711.98 & 2986.30 & 1402.08 & 1565.72 & 692.79 & 988.43 & 2327.97 & 2845.91 & 1267.53 & 1613.84 \\
Total energy grid [MWh] & 7.68 & 9.21 & 23.32 & 25.29 & 11.17 & 12.37 & 7.63 & 9.23 & 23.44 & 25.16 & 11.29 & 12.47 & 6.95 & 9.18 & 23.56 & 25.17 & 10.91 & 12.44 \\
\midrule
EVs per day & 902.6 & ... & 1857.8 & ... & 1097.1 & ... & 906.0 & ... & 1858.3 & ... & 1100.6 & ... & 900.6 & ... & 1854.9 & ... & 1097.5 & ... \\
Relative delta class [\%] & 52.02 & 56.27 & 53.76 & 56.76 & 53.28 & 56.20 & 51.83 & 56.04 & 54.07 & 56.65 & 53.45 & 56.26 & 50.45 & 56.21 & 53.79 & 56.62 & 52.54 & 56.32 \\
Fully charged EVs [\%] & 21.64 & 22.25 & 22.06 & 22.53 & 21.90 & 22.27 & 21.83 & 22.06 & 22.24 & 22.52 & 22.00 & 22.40 & 21.40 & 22.35 & 22.00 & 22.43 & 21.69 & 22.33 \\
2/3 charged EVs [\%] & 62.47 & 65.00 & 63.82 & 65.39 & 63.56 & 64.85 & 61.94 & 64.77 & 64.20 & 65.24 & 63.54 & 64.99 & 61.09 & 65.00 & 63.67 & 65.22 & 62.64 & 65.19 \\
Waiting time steps & 3.49 & 1.27 & 1.98 & 0.80 & 2.31 & 0.11 & 3.87 & 0.14 & 2.02 & 0.00 & 2.32 & 0.11 & 4.92 & 0.14 & 1.31 & 0.00 & 2.43 & 0.11 \\
\midrule
Total energy deliv. [MWh] & 13.36 & 14.90 & 29.01 & 30.97 & 16.84 & 18.08 & 13.31 & 14.91 & 29.13 & 30.86 & 16.97 & 18.17 & 12.63 & 14.86 & 29.23 & 30.85 & 16.57 & 18.14 \\
PV+Storage used [MWh] & 5.68 & 5.69 & 5.68 & 5.68 & 5.67 & 5.70 & 5.68 & 5.68 & 5.70 & 5.70 & 5.67 & 5.71 & 5.68 & 5.68 & 5.67 & 5.68 & 5.67 & 5.71 \\
Max $P_{\text{ntw}}$ reached [MW] & 1.96 & 2.20 & 2.05 & 2.06 & 2.20 & 2.19 & 1.97 & 2.20 & 2.08 & 2.04 & 2.19 & 2.19 & 2.05 & 2.20 & 2.19 & 2.05 & 2.17 & 2.19 \\
Total energy exch. [MWh] & 2.34 & 0.18 & 4.98 & 0.01 & 2.79 & 0.18 & 2.98 & 0.17 & 5.45 & 0.01 & 2.55 & 0.17 & 2.50 & 0.19 & 5.59 & 0.01 & 2.33 & 0.17 \\
Max energy exch. [kWh] & 87.90 & 87.90 & 91.67 & 0.23 & 84.31 & 20.41 & 91.39 & 14.31 & 91.67 & 0.10 & 91.67 & 19.64 & 78.74 & 15.96 & 91.61 & 0.19 & 83.11 & 20.57 \\
\bottomrule
\end{tabular}
}
\end{table*}

This appendix provides additional information on the validation simulation campaign and its results. All simulations were run in MATLAB R2023b on a Windows 13th Gen Intel(R) Core(TM) i9-13900H (2.60 GHz) and 32 GB of RAM, using CVX with the Mosek solver \cite{mosek}. One full shrinking horizon day of simulation took around 5-6 minutes to run, with a computational time lowering at every step $\tau$ of the optimization (see model \eqref{eq:mainLPmodel}), with the $T$-length simulation horizon shrinking progressively. Figure \ref{fig:time_shrinking} illustrates how runtime decreases as $\tau$ increases---equivalently, as the length  of the shrinking horizon decreases---using $\Delta = 5/60$ hours and $n=30$ as assumed. It should also be noted that computational time depends on the specific software, solver, and computer used for the simulation.

    \begin{figure}[h!] 
    \centering
    \includegraphics[scale=0.43]{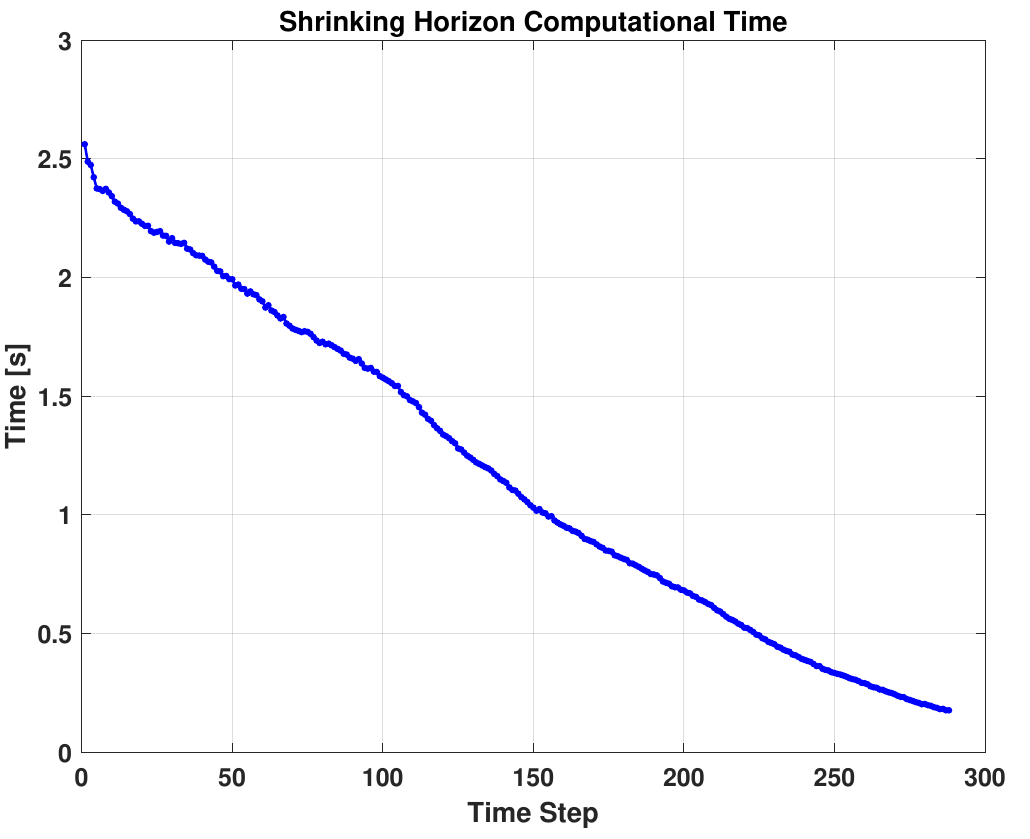} 
    \caption{Average running time (across 100 simulations) of a single optimization run at time step $\tau$ with the shrinking horizon.}
    \label{fig:time_shrinking}
    \end{figure}
    
The results of the nine-configurations validation campaign are summarized in Table \ref{tab:scenarios_big} in which we show, in addition to the average values of the metrics discussed in Section \ref{sec:validation_campaign_results}, average results of others evaluation metrics, highlighting again the economic improvement obtained by our proposed optimization framework compared to FIFS. The reported results are obtained by averaging the performance across all scenarios. Because the class-dependent departure process interacts with the charging policy through the state evolution, the total energy delivered over a day is not exactly the same under the two controllers; for this reason we report both absolute cost and cost per delivered kWh. All of the energy and power quantities were computed in kWh and kW in our simulations; in Table~\ref{tab:scenarios_big} we report some of them in MWh and MW for editing reasons. We observe that, in every configuration, Opt purchases less energy than FIFS, while still delivering a good charging service. This indicates that the proposed strategy is not merely shifting cost between components, but is actually improving the overall use of the available energy resources.

\newpage
\bibliographystyle{cas-model2-names}

\bibliography{bibliography1}

\end{document}